\documentclass[reqno, 11pt]{amsart}%
\usepackage{amssymb}
\usepackage[nesting]{hyperref}
\usepackage[pdftex]{graphicx}
\usepackage{listings}
\usepackage{multirow}
\usepackage{placeins}
\usepackage{color}
\usepackage{subfigure}
\usepackage{lscape}
\usepackage{dsfont}
\usepackage{amsmath}
\usepackage{amsfonts}%
\usepackage{tikz}
\usepackage{verbatim}
\providecommand{\U}[1]{\protect\rule{.1in}{.1in}}
\newcommand{\BlackBoxes}{\global\overfullrule5pt}

\BlackBoxes

\newcommand{\R}{\mathbb{R}}
\newcommand{\N}{\mathbb{N}}

\newcommand{\Eop}{\mathbb{E}}
\newcommand{\Pop}{\mathbb{P}}
\newcommand{\Q}{\mathbb{Q}}

\newcommand{\pr}[1]{#1=(#1_t)_{t\ge0}}
\newcommand{\seq}[1]{(#1_n)_{n\in\N}}
\newcommand{\alphabar}{{\bar{\alpha}}}
\newcommand{\mPP}[1]{(N^{1},\ldots,N^{#1})}

\newcommand{\mPPstpr}[1]{(N_t^{1},\ldots,N_t^{#1})_{t\ge0}}
\newcommand{\mPPpr}[1]{\mPP{#1}=\mPPstpr{#1}}
\newcommand{\nsum}{\sum_{n\in\N}}
\newcommand{\indicator}{{\mathds 1}}
\newcommand{\Ind}[1]{\indicator_{\{#1\}}}
\newcommand{\ind}[1]{\indicator_{#1}}
\newcommand*{\defeq}{\mathrel{\mathop{\raisebox{0.55pt}{\small$:$}}}=}

\newcommand{\PD}{{\cP(\D)}}
\newcommand{\cP}{{\mathcal P}}
\newcommand{\D}{\mathbb{D}}
\newcommand{\fF}{\mathfrak{F}}

\newcommand{\stpr}[1]{(#1_t)_{t\ge0}}
\newcommand{\cF}{{\mathcal F}}

\newcommand{\dif}{\textup{d}}
\newcommand{\TaS}{(N^1,\ldots,N^d)}
\newcommand{\idsum}{\sum_{i=1}^d}
\newcommand{\tseq}[3]{(#1_n,(#2_n,#3_n))_{n\in\N}}
\newcommand{\intEd}{\int_{E^d}}
\newcommand{\dsyz}{\dif s,\dif (y,z)}
\newcommand{\fG}{{\mathfrak G}}
\newcommand{\cG}{{\mathcal G}}
\newcommand{\Pas}{\P\text{-a.s.}}
\renewcommand{\P}{\mathbb{P}}
\newcommand{\prT}[1]{#1=(#1_t)_{t\in[0,T]}}
\newcommand{\Xxib}{X^{\xi,b}}
\newcommand{\xib}{(\xi,b)}
\newcommand{\UT}{\cU[0,T]}
\newcommand{\cU}{{\mathcal U}}
\newcommand{\dtyz}{\dif t,\dif (y,z)}
\newcommand{\UtT}{\cU[t,T]}

\newcommand{\Etx}{\Eop^{t,x}}
\newcommand{\intdalt}{\int_{(0,\infty)^d}}
\newcommand{\pbeta}{f_{\bar{\beta}}}
\newcommand{\DD}{{D\subset\D}}
\newcommand{\Dseq}[1]{(#1_D)_{D\subset\D}}

\newcommand{\stpralt}[1]{(#1(t))_{t\ge0}}
\newcommand{\DBsum}{\sum_{D\in B}}
\newcommand{\fracbetaqt}{\frac{\beta_D+q_D(t)}{\normbetaqt}}
\newcommand{\PPD}{{\cP(\cP(\D))}}
\newcommand{\normbetaqt}{\|\bar{\beta}+q_t\|}
\newcommand{\Dsum}{\sum_{D\subset\D}}
\newcommand{\domainVDir}{[0,T]\times\R\times\Delta_m\times\N_0^\ell}
\newcommand{\fracbetaq}{\frac{\beta_D+q_D}{\normbetaq}}
\newcommand{\normbetaq}{\|\bar{\beta}+q\|}
\newcommand{\domaingDir}{[0,T]\times\Delta_m\times\N_0^\ell}
\newcommand{\Esum}{\sum_{E\subset\D}}
\newcommand{\Eprod}{\prod_{E\subset\D}}

\newcommand{\EDDprod}{\prod_{E\subset\D\setminus\{D\}}}
\newcommand{\prTalt}[1]{#1=(#1(t))_{t\in[0,T]}}
\newcommand{\partialvar}[1]{\frac{\partial}{\partial #1}}

\newcommand{\dprtTalt}[2]{(#1,#2)=(#1(s),#2(s))_{s\in[t,T]}}

\newcommand{\stprT}[1]{(#1_t)_{t\in[0,T]}}
\newcommand{\stT}{{s\in[t,T]}}
\newcommand{\stprtT}[1]{(#1_s)_{s\in[t,T]}}

\newcommand{\Xxibstar}{X^{\xi^\star,b^\star}}

\newcommand{\fracbetaCt}{\frac{\beta_D+q_D(t)}{\normbetaqt}}

\newcommand{\fracbetaqs}{\frac{\beta_D+q_D(s)}{\normbetaqs}}
\newcommand{\fracbetaCs}{\frac{\beta_D+q_D(s)}{\normbetaqs}}
\newcommand{\normbetaqs}{\|\bar{\beta}+q_s\|}
\newcommand{\etaxibhat}{\hat{\eta}^{\xi,b}}
\newcommand{\etaxibbar}{\bar{\eta}^{\xi,b}}
\newcommand{\etaxibtilde}{\tilde{\eta}^{\xi,b}}

\newcommand{\qtilde}{{\tilde q}}
\newcommand{\fracbetaD}{\frac{\beta_D}{\|\bar{\beta}\|}}

\newtheorem{theorem}{Theorem}
\newtheorem{corollary}[theorem]{Corollary}
\newtheorem{lemma}[theorem]{Lemma}
\newtheorem{proposition}[theorem]{Proposition}
\theoremstyle{definition}

\newtheorem{definition}[theorem]{Definition}
\numberwithin{equation}{section} \numberwithin{theorem}{section}
\def\0{\kern0pt\-\nobreak\hskip0pt\relax}

\makeatletter
\AtBeginDocument{ \def\@serieslogo{ \vbox to\headheight{ \parindent\z@ \fontsize{6}{7\p@}\selectfont
\vss}}}
\def\overbar#1{\skewbar{#1}{-1}{-1}{.25}}
\def\skewbar#1#2#3#4{{\overbarpalette\makeoverbar{#1}{#2}{#3}{#4}}#1}
\def\overbarpalette#1#2#3#4#5{\mathchoice
{#1\textfont\displaystyle{#2}1{#3}{#4}{#5}}
{#1\textfont\textstyle{#2}1{#3}{#4}{#5}}
{#1\scriptfont\scriptstyle{#2}{.7}{#3}{#4}{#5}}
{#1\scriptscriptfont\scriptscriptstyle{#2}{.5}{#3}{#4}{#5}}}
\def\makeoverbar#1#2#3#4#5#6#7{ \setbox0=\hbox{$\m@th#2\mkern#5mu{{}#3{}}\mkern#6mu$} \setbox1=\null \dimen@=#4\fontdimen8#13 \dimen@=3.5\dimen@
\advance\dimen@ by \ht0 \dimen@=-#7\dimen@ \advance\dimen@ by \wd0
\ht1=\ht0 \dp1=\dp0 \wd1=\dimen@
\dimen@=\fontdimen8#13 \fontdimen8#13=#4\fontdimen8#13
\rlap{\hbox to \wd0{$\m@th\hss#2{\overline{\box1}}\mkern#5mu$}}
\fontdimen8#13=\dimen@}
\def\mylabel#1#2{{\def\@currentlabel{#2}\label{#1}}}
\makeatother
\newcommand{\Fbar}{\overbar{F}}

\begin{document}
\title[Robust Optimal Investment and Reinsurance Problems with Learning]{Robust Optimal Investment and Reinsurance Problems with Learning }
\author[N. \smash{B\"auerle}]{Nicole B\"auerle${}^*$}
\address[N. B\"auerle]{Department of Mathematics,
Karlsruhe Institute of Technology (KIT), D-76128 Karlsruhe, Germany}

\email{\href{mailto:nicole.baeuerle@kit.edu}
{nicole.baeuerle@kit.edu}}

\author[G. \smash{Leimcke}]{Gregor Leimcke${}^*$}
\address[G. Leimcke]{Department of Mathematics,
Karlsruhe Institute of Technology (KIT), D-76128 Karlsruhe, Germany}

\email{\href{gregor.leimcke@kit.edu} {gregor.leimcke@kit.edu}}

\thanks{${}^*$ Department of Mathematics,
Karlsruhe Institute of Technology (KIT), D-76128 Karlsruhe, Germany}
%\thanks{${}^\dagger$ Department of Mathematics, Karlsruhe Institute of Technology (KIT), D-76128 Karlsruhe, Germany}
\begin{abstract}
In this paper we consider an optimal investment and reinsurance problem with partially unknown model parameters which are allowed to be learned. The model includes multiple business lines and dependence between them. The aim is to maximize the expected exponential utility of terminal wealth which is shown to imply a robust approach. We can solve this problem using a generalized HJB equation where derivatives are replaced by generalized Clarke gradients. The optimal investment strategy can be determined explicitly and the optimal reinsurance strategy is given in terms of the solution of an equation. Since this equation is hard to solve, we derive bounds for the optimal reinsurance strategy via comparison arguments.

\end{abstract}
\maketitle

%\listoftodos

\makeatletter \providecommand\@dotsep{5} \makeatother
%\listoftodos[Changes in Orange/Red To Do List in Green / Blue]\relax

%\address[N. B\"auerle]{Department of Mathematics,
%Karlsruhe Institute of Technology, D-76128 Karlsruhe, Germany}

%\email{\href{mailto:nicole.baeuerle@kit.edu}{nicole.baeuerle@kit.edu}}

\vspace{0.5cm}
\begin{minipage}{14cm}
{\small
\begin{description}
\item[\rm \textsc{ Key words} ]
{\small Risk Theory, Stochastic Control, Filter, Robust Approach}
%\item[\rm \textsc{AMS subject classifications}]
%{\small  ..... }
\end{description}
}
\end{minipage}

\section{Introduction}
The insurance industry is currently facing a variety of challenges. On the one hand, the number and amount of insurance losses are growing caused by an increasing frequency of weather extremes due to climate change (see \cite{MunichRe2019}). On the other hand, the current structural low interest rate environment and higher volatility on the financial markets make it more difficult to achieve profitable investments. These challenges call for effective strategies to reduce insurance risk and to optimize capital investments and have  attracted interest from researches in actuarial mathematics for many years. In fact, a classical task in risk theory is to deal with optimal risk control and optimal asset allocation for an insurance company. Such problems have been intensively studied in literature using various optimization criteria, where maximizing the utility and minimizing the probability of ruin are the two main optimization criteria (see e.g.\ \cite{Schmidli2002} and references given there).%, \cite{PromislowYoung2005} and \cite{CaoWan2009}).
 
%Papers using a Black-Scholes-type financial market and a proportional reinsurance (as it will be used in this work) are  \cite{Schmidli2002}, \cite{PromislowYoung2005} and \cite{CaoWan2009}. While the first two articles provide optimal investment and reinsurance strategies (closed-form and analytical expression for the reinsurance strategy, respectively) under the criteria of minimizing the ruin probability, the third paper offers an explicit expression for the problem of maximizing the exponential utility of terminal wealth. 

However, in most articles, the assumption of full information is used as a common feature, which means that the insurer has complete knowledge of the model. However, in reality, insurance companies operate in a setting with partial information. That is, with regard to the net claim process, only the claim arrival times and magnitudes are directly observable; the claim intensity, which is required by all net claim models, is not observable by the insurer as pointed out in \cite[Ch.\,2]{Grandell1991}. Therefore we study the optimal investment and reinsurance problem in a partial information framework. More precisely we consider a Bayesian approach which means that we allow for learning unknown model parameters. On the other hand we use an exponential utility function as optimization criterion which can be interpreted as a robust approach. 

There are quite a number of papers on robust decision making in actuarial sciences, in particular for optimal reinsurance and investment, see e.g.\ \cite{ZS09,ZZS16,GVY17,GVY18} among others. But all the approaches so far consider a classical optimization problem like utility maximization under alternative models given in form of different probability measures. In this paper indeed we explain that the exponential utility can be interpreted as a robust control approach, thus avoiding a second complicated optimization. 

A paper with partial information is e.g.\ \cite{LiangBayraktar2014}. Based on the suggestion in \cite[p.\,165]{AsmussenAlbrecher2010}, the authors there consider the optimal investment and reinsurance problem for maximizing exponential utility under the assumption that the claim intensity and loss distribution depend on the states of a non-observable Markov chain (hidden Markov chain), which describes different states of the environment, whereby the net claim process is modelled as compound Poisson process and the fully observable financial market is modelled as Black-Scholes financial market with one risky and one risk-free asset. 

%In this thesis, we use the same financial market model and our assumption on the claim intensity can be considered as the special case of one state of the above mentioned Markovian regime-switching model; namely, we model the intensity as an unobservable random variable, which puts us in a Bayesian setting.
However, the sparse literature with  partial information focuses on just one line of business to gain an optimal reinsurance strategy. But in reality there is often a dependence between different risk processes of an insurance company. This results from the fact that the customers of a typical insurance company have insurance policies of different types such as building, private liability or health insurance contracts.
A simplified example of a possible dependence between several types of risk is that of a storm event accompanied by heavy rainfall where flying roof tiles cause damages to third parties and flooding leads to building damages. 
%In addition to this dependence between private liability insurance and building insurance, there may even be a dependence on motor liability insurance and health insurance if a car accident occurs caused by adverse traffic circumstances due to that heavy rainfall. 
Therefore, to model the insurance risks of an insurer appropriately, we need to capture the dependence structure using a multivariate model.
 
A commonly used approach to impose dependence between several types of insurance risks is accomplished by thinning, which is also the case in this paper. 
The idea of this approach is that the occurrence of claims depends on a certain process which generates events that cause damages of the line of business $i$ with probability $p_i$ and of the line of business $j$ with probability $p_j$, where all caused claims occur simultaneously at the trigger arrival time.
Therefore these models are referred to as \emph{common shock risk models}.  An example of a shock event is the above-described storm event. Typically the corresponding claim sizes are determined independently of the appearance times (see e.g. \cite{BaeuerleGruebel2005}). 

Another multivariate model that avoids a reference to an external mechanism is given in \cite{BaeuerleGruebel2008}, where the authors propose a multivariate continuous Markov chain of pure birth type with interdependency arising from dependences of the birth rates on the number of claims in other component processes.
In \cite{SchererSelch2018}, the dependence of the marginal processes of a multiple claim arrival process is constructed  by introducing a L\'evy subordinator serving as a joint stochastic clock, which leads to a multivariate Cox process in the sense that the marginal processes are univariate Cox processes. 
In connection with optimal reinsurance problems, a L\'evy approach is discussed in \cite{bb11}.
There, the authors have shown that a constant investment and reinsurance strategy (proportional reinsurance as well as the mixture of proportional and excess-of-loss reinsurance) is the optimal strategy for maximizing the exponential utility of terminal wealth.

In addition to the L\'evy model, optimization problems with common shock models have been investigated in \cite{Centeno2005}, where  optimal excess-of-loss retention limits are studied for a bivariate compound Poisson risk model in a static setting. The corresponding dynamic model was used in \cite{BaiCaiZhou2013} to derive optimal excess-of-loss reinsurance policies (which turns out to be constant) under the criterion of minimizing the ruin probability making use of a diffusion approximation.
For the same model, \cite{LiangYuen2016} have derived a closed-form expression for the optimal proportional reinsurance strategy of the exponential utility maximizing problem both with and without diffusion approximation by using the variance premium principle. 
In the presence of a Black-Scholes financial market, the same problem has been investigated in \cite{BiChen2019} with the expected value premium principle. 
For the case of an insurance company with more than two lines of business, \cite{YuenLiangZhou2015} and \cite{WeiLiangYuen2018} seek optimal proportional reinsurance to maximize the exponential utility of terminal wealth and the adjustment coefficient, respectively, where the strategies are only stated for two classes of business.
However, all optimization problems with multivariate insurance models are considered under full information.

We will describe the dependence structure between different lines of business by the thinning approach while we deal with unobservable thinning probabilities.
To the authors' knowledge, this is the first time that an optimal reinsurance and investment problem under partial information using a multivariate claim arrival model with possibly dependent marginal processes is studied.
To solve the optimal control problem, the dynamic programming approach will be applied. However since the value function may not be differentiable, a generalization using the Clarke gradient will be applied (this idea has been used before in \cite{BaeuerleRieder2007}, \cite{LiangBayraktar2014}). 

The outline of our paper is as follows: We introduce the partial information problem under the assumptions of observable claim size distribution, unobservable background intensity taking values in a finite set and Dirichlet distributed thinning probabilities in Section \ref{sec:model}. We also explain the robust approach which is inherit in the criterion of maximizing exponential utility. Using a filter as an estimator for the background intensity and the conjugated property of the Dirichlet distribution, we proceed in Section \ref{sec:learn} by stating the reduced control problem.   The corresponding generalized Hamilton Jacobi Bellman (HJB) equation will be introduced in Section \ref{sec:sol}, where we need to replace partial derivatives w.r.t.\ the time and the components of the filter for the background intensity by the corresponding Clarke gradient. 
The HJB equation yields the same optimal investment strategy as in the classical Merton problem and the optimal reinsurance strategy has to be characterized implicitly. Next we prove a verification theorem  and show the existence and optimality of the proposed  strategy. Finally in Section \ref{sec:comp}, we provide a comparison result  under the assumption of identical claim size distributions for all insurance classes, which is visualized by some examples in the last section.

\section{The Optimal Investment and Reinsurance Model}\label{sec:model}
We consider an insurance company with several lines of business. The aim is to maximize the expected utility of the terminal surplus of the considered insurance company by choosing optimal investment and reinsurance strategies.  For the moment we assume that all model data is known.
\subsection{The claim arrival process}
In the following, let $d\in\N$ be the number of business lines of the insurer. The claim arrival model is a  Poisson process $\pr{N}$ with intensity $\lambda$.  We interpret the arrival times of the  Poisson process $N$, denoted by $\seq{T}$, as events which trigger various kinds of insurance claims. The lines of business which are affected by the trigger event at $T_n$ are given by a random variable $Z_n$ with values in $\mathcal{P}(\mathbb{D})$ (power set of  $\mathbb{D}=\{1,\ldots,d\}$). We assume that $(Z_n)_{n\in\N}$ are i.i.d.\ with $\Pop(Z_n=D)=\alpha_D, D\subset \mathbb{D}$ and denote $\bar{\alpha}=(\alpha_D)_{D\subset\mathbb{D}}$. The interdependencies between the lines of business are fully determined by $\alphabar$. We call the components of $\alphabar$ \emph{thinning probabilities} since they thin the trigger arrival times. Moreover, w.l.o.g.\ $\Pop(Z_1=\emptyset)=0$, i.e.\ every shock event leads to at least one insurance damage. Otherwise we could reduce the intensity of $N$. Therefore the \emph{(multivariate) claim arrival process}, denoted by $\mPPpr{d}$, is defined by
\begin{equation*}
N_t^i= \nsum \Ind{T_n\le t} \Ind{i\in Z_n},\quad t\ge 0, \quad i=1,\ldots,d.
\end{equation*}
So $(N^1,\ldots,N^d)$ is a $d$-dimensional counting process, where $N_t^i$ counts the number of claims of the $i$th business line up to time $t$.    
The \emph{claim sizes} are described by a $d$-dimensional sequence $\seq{Y}$ with $Y_n=(Y_n^1,\ldots,Y_n^d)$ of i.i.d. $(0,\infty)^d$-valued random variables with distribution $F$. It is worth to note that the claims sizes from various business lines can be dependent. We assume that $Y_1,Y_2,\ldots$ are independent of the sequences $\seq{T}$ and $\seq{Z}$.
The sum of the claim sizes of all $d$ insurance classes which appear at the arrival times of the multivariate claim arrival process $\TaS$ up to time $t$, denoted by $\pr{S}$, gives the aggregated claim amount process, i.e.\ 
% So the \emph{aggregated claim amount process} or \emph{total claim amount process}, denoted by $\pr{S}$, 
it is given by 
\begin{equation*}
S_t = \idsum \nsum Y^i_n\,\Ind{T_n\le t}\,\Ind{i\in Z_n},\quad t\ge0.
\end{equation*}
From now on, we set $\Psi\defeq\tseq{T}{Y}{Z}$ and let $E := (0,\infty)^d\times\PD$.  That is, $\Psi$ is the $E$-Marked Point Process which contains the information of the claim arrival times, the thinning sequence and the claim sizes. The filtration generated by $\Psi$ is denoted by $\fF^\Psi=(\mathcal{F}_t^\Psi)_{t\ge 0}$ and the intensity measure of $\Psi$ is given by $\nu(t,(A,B))=\lambda  F(A) \sum_{D\in B}\alpha_D. $
Using the introduced Marked Point Process $\Psi$, we can write
\begin{equation}\label{eq:aggclaim}
S_t = \int_0^t \int_E \idsum y_i\,\ind{z}(i)\,\Psi(\dsyz),\quad t\ge0.
\end{equation}
It should be noted that the aggregated claim amount process $S$ is observable for the insurance company and thus the natural filtration of $\Psi$, denoted by  $\fF^\Psi$, is known by the insurer. We can interpret $S$ given by~\eqref{eq:aggclaim} as the aggregated claim amount process of a heterogeneous insurance portfolio where the random elements $Z_n$ yield the information of which type  the claim size distribution  of the claim at time $T_n$ is.
Finally we need the following assumption on the integrability of the claim size distribution:
\begin{equation}\label{eqDir:F}
 M_F(z) := \Eop \bigg[\exp\bigg\{z\idsum Y_1^i\bigg\}\bigg]  <\infty,\quad z\in\R.
\end{equation}

\subsection{The financial market} 
The surplus will be invested by the insurer into a financial market, which will be modelled as a classical Black-Scholes market. 
So it is supposed that there exists one risk-free asset and one risky asset. The price process of the \emph{risk-free asset}, denoted by $\pr{B}$, is given by
\begin{equation*}
d B_t = rB_t dt, \quad B_0=1,
\end{equation*}
where $r\in\R$ denotes the \emph{risk-free interest rate}. That is, $B_t = e^{rt}$ for all $t\ge0$.
The price process of the \emph{risky asset}, denoted by  $\pr{P}$, is given by
\begin{equation*}
d P_t = \mu P_t dt + \sigma P_t d W_t, \quad P_0=1,
\end{equation*}
where $\mu\in\R$ and $\sigma>0$ are constants describing the drift and volatility of the risky asset, respectively, and $\pr{W}$ is a standard Brownian motion.  We assume that the Brownian motion $W$ is independent of $\seq{T}$, $\seq{Y}$ and $\seq{Z}$. We denote by $\stpr{\cF^W}$ the augmented Brownian filtration of $W$.     Throughout this work, $\fG=\stpr{\cG}$ denotes the observable filtration of the insurer which is given by
\begin{equation*}
\cG_t = \cF^W_t\vee\cF_t^\Psi,\quad t\ge0.
\end{equation*}
\subsection{The strategies}
We assume that the wealth of the insurance company is invested into the previously described financial market.
%The insurer can choose the amount of its surplus that is invested at time $t$ into the risky asset $P$, where we also allow short-sell by the insurer which is represented by a negative volume put into the risky asset. In addition, the insurance company is permit to lend and borrow, respectively, an infinitesimal amount of money.
\begin{definition}\label{def:investment}
	An \emph{investment strategy}, denoted by $\pr{\xi}$, is an $\R$-valued, bounded, c\`{a}dl\`{a}g and $\fG$-progressively measurable process.% with %$\int_0^t |\xi_s|^2\ds < \infty$ $\Pas$ for all $t\ge0$. %\marginnote{Its formulated for $I\subset[0,\infty)$ since later an investment strategy will be consider on $[t,T]$, see \eqref{eq:Atpdis}.}
%	\begin{equation}\label{eq:invstrcond}
%	\int_0^t |\xi_s|^2 ds < \infty \quad \Pas\quad \text{for all }t\ge0.
%	\end{equation}
\end{definition}
Note that for simplicity we assume here bounded strategies, i.e.\ $|\xi_t|\le K,$ for $K>0$. We will later see that this is no restriction. Further, we assume that the first-line insurer has the possibility to take a proportional reinsurance. Therefore, the \emph{part of the insurance claims paid by the insurer}, denoted by $h(b,y)$, satisfies
\begin{equation*}
h(b,y) = b\cdot y
\end{equation*}
with \emph{retention level} $b\in[0,1]$ and \emph{insurance claim} $y\in(0,\infty)$. Here we suppose that the insurer is allowed to reinsure a fraction of her/his claims with retention level $b_t\in[0,1]$ at every time $t$. 

\begin{definition}\label{def:reinsurance}
	A \emph{reinsurance strategy}, denoted by $\pr{b}$, is a $[0,1]$-valued, c\`{a}dl\`{a}g and $\fG$-predictable process. 
\end{definition}
We denote by $\UtT$ the set of all admissible strategies $(\xi,b)$ on $[t,T]$. 
%For simplicity we consider bounded strategies only. We will later see that this is no restriction.
%Of course, sharing risk by ceding proportions of claims to a reinsurer reduces the premium income of the first-line insurer. To discuss this reduction in detail, w
We assume that the policyholder's payments to the insurance company are modelled by a fixed \emph{premium (income) rate} $c=(1+\eta)\kappa$ with safety loading $\eta>0$ and fixed constant $\kappa>0$, which means that premiums are calculated by the expected value principle.
If the insurer chooses retention levels less than one, then the insurer has to pay premiums to the reinsurer. The \emph{part of the premium rate left to the insurance company} at retention level $b\in[0,1]$, denoted by $c(b)$, is $c(b) = c - \delta(b)$, where $\delta(b)$ denotes the \emph{reinsurance premium rate}. We say $c(b)$ is the \emph{net income rate}.
% In the case of no reinsurance (retention level of 1), the et income rate is $c(1)=c$. 
Moreover, the net income rate $c(b)$ should increase in $b$, which is fulfilled by setting $\delta(b) \defeq  (1-b)(1+\theta)\kappa$ with $\theta>\eta$ which represents the safety loading of the reinsurer.
Therefore
\begin{equation}\label{eq:cb}
c(b) = (1+\eta)\kappa - (1-b)(1+\theta)\kappa = (\eta-\theta)\kappa + (1+\theta)\kappa\,b,
\end{equation}
where $\eta-\theta<0$.
This reinsurance premium model is used e.g.\ in \cite{ZhuShi2019}. 
The surplus process $\prT{\Xxib}$ under an admissible investment-reinsurance strategy $\xib\in\UT$ is given by
\begin{align*}
d \Xxib_t
&=(\Xxib_t - \xi_t)r dt + \xi_t(\mu dt+\sigma dW_t) + c(b_t)dt - b_t dS_t  \\
&= \left(r\Xxib_t + (\mu - r)\xi_t + c(b_t)\right) dt + \xi_t\sigma dW_t - b_{t}dS_t.
\end{align*}
We suppose that $\Xxib_0=x >0$  is the initial capital of the insurance company. An alternative representation of the surplus process with the help of the random measure will be  useful. The dynamics of the surplus can for $t\ge 0$  be written as
\begin{equation}\label{eq:surplus}
\begin{aligned}
\dif \Xxib_t &= \left(r X_t^{\xi,b} + (\mu-r)\xi_t + c(b_t)\right)\dif t+ \xi_t\sigma\dif W_t  - \int_{E} b_t\idsum y_i\ind{z}(i)\,\Psi(\dtyz).
\end{aligned}
\end{equation}

\subsection{The optimization problem}
Clearly, the insurance company is interested in an optimal investment-reinsurance strategy. But there are various optimality criteria to specify optimization of proportional reinsurance and investment strategies.  We consider the expected utility of wealth at the terminal time $T$ as criterion with
%So we assume that the insurer has an 
exponential utility function $U:\R\to\R$ 
\begin{equation}\label{eq:u}
U(x)=-e^{-\alpha x},
\end{equation}
where the parameter $\alpha>0$ measures the \emph{degree of risk aversion}. The choice of this criterion will be justified below.
Next, we are going to formulate the dynamic optimization problem. 
We define the value functions, for any $(t,x)\in[0,T]\times\R$ and $(\xi,b)\in\UtT$, by
\begin{equation}\label{eq:problem} %\tag{P}
\begin{aligned}
V^{\xi,b}(t,x) &\defeq \Etx\big[U(\Xxib_T)\big],  \\
{V}(t,x) &\defeq \sup_{(\xi,b)\in\UtT}V^{\xi,b}(t,x),
\end{aligned}
\end{equation}
where the expectation $\Etx$ is taken w.r.t.\ the conditional probability measure $\P^{t,x}$ where $\Xxib_t=x$ is given (when $t=0$ we simply write $\Eop^x$). This optimization criterion has an interesting interpretation. Instead of $V^{\xi,b}(0,x)$ we can equivalently maximize $-\frac1\alpha \log \Eop^x\big[e^{-\alpha \Xxib_T}\big]$ which is the entropic risk measure of terminal wealth. For small $\alpha$ this is approximately equal to  (see e.g. \cite{bp}, \cite{br19})
$$-\frac1\alpha \log \Eop^x\big[e^{-\alpha \Xxib_T}\big] \approx  \Eop^x\big[ \Xxib_T\big]-\frac12 \alpha Var^x(\Xxib_T). $$ 
Thus, for small $\alpha>0$ we maximize the expectation penalized by the variance. This is a risk-sensitive criterion on one hand, but can also be seen as the Lagrange-function of a mean-variance problem. Another interesting feature of this criterion is that it has a dual representation as
$$ -\frac1\alpha \log \Eop^x\big[e^{-\alpha \Xxib_T}\big] = \inf_{\mathbb{Q}\ll \Pop^x} \Big( \Eop^{\mathbb{Q}}\big[ \Xxib_T\big] +\frac1\alpha I(\mathbb{Q} \| \Pop^x) \Big)$$  for r.v.\ $\Xxib_T$ which are bounded from above with
$$ I(\mu\| \nu) := \left\{ \begin{array}{cc}
	\int\ln(\frac{d \mu}{d \nu})d \mu, & \mbox{ if } \mu\ll \nu,\\
	\infty, & \mbox{ otherwise},
\end{array}  \right. $$  for the relative entropy function or Kullback-Leibler distance $I$ between two probability measures  $\mu,\nu$ (see e.g. \cite{DPMR}).  From this representation we see that the  case $\alpha\uparrow \infty$ corresponds to the case of a robust optimization problem or worst-case optimization problem where the insurer maximizes the surplus if nature chooses the least favourable measure for the model. For $\alpha>0$ this means that potentially a whole range of beliefs about $\Pop^x$ is considered but deviations from the baseline model $\Pop^x$ are penalized. In some cases this yields an alternative method to solve the optimization problem. Let us for example consider the classical (one-dimensional) Cram\'er-Lundberg model with reinsurance which is a special case of our model. The surplus process is given by $$ dX_t = c(b_t)dt - \int_{(0,\infty)} b_t y \Psi(dt,dy).$$ It is well-known that the worst-case probability measure in this representation is also equivalent to $\Pop^x$. So we can restrict  our search of the worst-case measure to measures with a density of the form (see e.g. \cite{J75} \cite{sk75}),  
\begin{equation}
 \frac{d \Q}{d \Pop^x}(\omega) = \exp\Big(\int_0^t \int_{(0,\infty)} \ln\bigg(\frac{g(\omega,s,dy)}{\lambda F(dy)}\bigg)\Psi(\omega,ds,dy) 
- \int_0^t \int_{(0,\infty)} [g(\omega,s,dy)-\lambda F(dy)] ds\Big),
\end{equation}
where $$ \int_0^t \int_B b_s y \Psi(\omega,ds,dy) - \int_0^t b_s y g(\omega,s,B)ds $$ is a $\fG$-martingale under $\Q$. Thus, we can parametrize $\Q$ by the random intensity measure $g$.
Hence
\begin{eqnarray}\nonumber
 &&\Eop^{\mathbb{Q}}\big[ \Xxib_T\big] +\frac1\alpha I(\mathbb{Q} \| \Pop^x) \\ \nonumber&=& x+\Eop^\Q\Big[\int_0^t c(b_s) ds - \int_0^t  \int_{(0,\infty)} b_s y \Psi(ds,dy)  + \frac{1}{\alpha} \int_0^t \int_{(0,\infty)} \ln\bigg(\frac{g(s,dy)}{\lambda F(dy)}\bigg)\Psi(ds,dy)\\\nonumber&& -\frac{1}{\alpha} \int_0^t \int_{(0,\infty)} [g(s,dy)-\lambda F(dy)] ds\Big]\\ \nonumber
 &=&  x+\Eop^\Q\Big[\int_0^t c(b_s) ds - \Big( \int_0^t  \int_{(0,\infty)} b_s y g(s,dy) - \frac{1}{\alpha}  \ln \bigg(\frac{g(s,dy)}{\lambda F(dy)}\bigg)g(s,dy)\\
 && +\frac{1}{\alpha}  g(s,dy)-\frac{\lambda}{\alpha} F(dy) ds\Big)\Big]. \label{eq:Qgf}
\end{eqnarray}
Minimizing this expression w.r.t.\ $g$ can be done pointwise and  yields by simple differentiation that the worst case  measure is given by
$$g(\omega,s,dy)= \lambda \exp(\alpha b_s(\omega) y) F(dy).$$ Plugging this expression in \eqref{eq:Qgf} yields
$$  x+\Eop^\Q\bigg[\int_0^t c(b_s) ds - \frac{\lambda}{\alpha} \int_0^t  \int_\R  \exp(\alpha b_s y) F(dy) ds +\frac{\lambda}{\alpha} t \bigg].$$ 
and maximizing for $(b_s)$ which can again be done pointwise finally gives the Euler equation
\begin{equation}\label{eq:Euler}
 c'(b)-\lambda \int y e^{\alpha b} F(dy)=0
 \end{equation} for the first-order condition implying the optimality of $(b_t)_{t\ge0}$. The optimal strategy is here given by $b_s^*=b$ with $b$ solving \eqref{eq:Euler} (see also \cite{bb11}, \cite{LiangBayraktar2014}).
This discussion illustrates that the exponential utility function is really useful on one hand since by choosing $\alpha$ we can interpolate between a risk-sensitive criterion and a robust point of view. Moreover, it is still tractable as we will see in the next sections. However, the robust approach is often too pessimistic and we want to combine it with learning model parameters. That's why we further generalize the model in the next section.

\section{A Model with Learning}\label{sec:learn}
Now we assume that the precise parameters $\lambda$ and $\bar{\alpha}$ of the model are not known. Instead we take a Bayesian approach and suppose that $\lambda$ is a realization of a random variable $\Lambda$ which takes values in a set $\{\lambda_1,\ldots ,\lambda_m\}$ and has initial distribution $\pi_\Lambda(j) = \Pop(\Lambda=\lambda_j), j=1,\ldots,m.$ 
W.l.o.g.\  $0<\lambda_1<\ldots <\lambda_m$. Moreover, we assume that   the initial distribution of $\bar{\alpha}$ is a Dirichlet distribution with parameter $\bar{\beta}=(\beta_D)_{D\subset\mathbb{D}}\in(0,\infty)^\ell$, where $\ell=2^d-1$.
\begin{definition}[Dirichlet distribution; \cite{DeGroot1970}, p.\,49]\label{def:Dir}
	A random vector $X=(X_1,\ldots,X_\ell)$ has a \emph{Dirichlet distribution} with parameter  $\bar{\beta}=(\beta_1,\ldots,\beta_\ell)\in(0,\infty)^\ell$, if the probability density function $\pbeta(\cdot)$  is given by
	\begin{equation*}
	\pbeta(x) = \frac{\Gamma(\beta_1+\ldots+\beta_\ell)}{\Gamma(\beta_1)\cdot\ldots\cdot\Gamma(\beta_\ell)} \prod_{i=1}^\ell x_i^{\beta_i-1},\quad x=(x_1,\ldots,x_\ell)\in\mathring{\Delta}_\ell,
	\end{equation*}
	where $\mathring{\Delta}_\ell$ denotes the interior of $\Delta_\ell :=\big\{x\in\R_+^\ell: \sum_{i=1}^{\ell}x_i=1\big\}$ and $\Gamma$ the gamma function.
   We shortly write $X\sim Dir(\bar{\beta})$.
\end{definition}
Thus, we allow that model parameters are learned by observing the process. In what follows we define 
\begin{equation}\label{eq:qD}
q_D(t) := \sum_{i=1}^{N_t}1_{\{Z_i=D\}}
\end{equation} and $q_t := \big(q_D(t)\big)_\DD$. Thus $\pr{q}$ is an $\N_0^\ell$-valued process and $q_D(t)$ counts the number of realizations of $Z_n$ equal to $D$ up to time $t$. 

The reason for the choice of the Dirichlet distribution as prior is the conjugated property of the Dirichlet prior, which is stated next.

\begin{theorem}[\cite{DeGroot1970},\,Thm\,9.8.1]\label{th:Dirconjugated}
	The posterior distribution of $\alphabar$ given $q_t=c$ with $c=\Dseq{c}\in\N_0^\ell$ is a Dirichlet distribution  with parameter vector $\bar{\beta}+c=(\beta_D+c_D)_\DD$.
\end{theorem}

It should also be noted that the marginal distribution of the $j$th component of a Dirichlet-distributed random vector $(X_1,\ldots,X_\ell)$ with parameter $\bar{\beta}\in(0,\infty)^\ell$, is Beta distributed with parameters $\beta_j$ and $\sum_{i=1}^\ell\beta_i - \beta_j$, compare \cite[p.\,50]{DeGroot1970}. In particular $\Eop X_j=\frac{\beta_j}{\sum_{i=1}^{\ell} \beta_i}$. This fact implies immediately the following result.

\begin{corollary}\label{co:Dirconjugated}
	The posterior distribution of $\alpha_D$ given $q_t=c$ with $c=\Dseq{c}\in\N_0^\ell$ is a Beta distribution with parameters $\beta_D+c_D$ and $\sum_{E\subset \mathbb{D}, E\neq D}(\beta_{E}+c_{E})$. 
\end{corollary}

\subsection{Filtering}
Since the strategies have to be $\fG$-predictable and  $\fG$-progressively measurable respectively, the task is to reduce the partially observable control problem~\eqref{eq:problem} within the introduced framework to one with a state process that describes the available information about the unknown background intensity and interdependencies between the line of business. The conditional distribution of $\bar{\alpha}$ can immediately be derived from $(q_t)$ with Theorem \ref{th:Dirconjugated}. For the background intensity we determine a filter process.
 Throughout this paper, we denote by $(\hat{\Lambda}_t)_{t\ge 0}$ the c\`{a}dl\`{a}g modification of the process $(\Eop[\Lambda|\cG_t])_{t\ge0}$
 and we write 
\begin{equation}\label{eq:filterdef}
p_j(t) = \Pop(\Lambda=\lambda_j \mid \cG_t),\quad t\ge0.
\end{equation}
Moreover, we denote by $\pr{p}$ the $m$-dimensional process defined by
\begin{equation*}
p_t \defeq (p_1(t),\ldots,p_m(t)),\quad t\ge0.
\end{equation*}

The following result provides the dynamics of the \emph{filter process} $\stpr{p}$. It is a standard result and can be found e.g.\ in \cite{bre}, \cite{BaeuerleRieder2007}.

\begin{theorem}\label{thDir:filp}
	For any $j\in\{1,\ldots,m\}$, the process $\stpralt{p_j}$ satisfies
	\begin{equation}\label{eqDir:pj}
	p_j(t) = \pi_\Lambda(j) + \int_0^t \bigg(\frac{\lambda_j\,p_j(s-)}{\hat{\Lambda}_{s-}}-p_j(s-)\bigg)dN_s
	+  \int_0^t \,p_j(s)\big({\hat{\Lambda}_{s}}-{\lambda_j}\big)ds,\quad t\ge0.
	\end{equation}
\end{theorem}
 Note that $(p_t)$ is a piecewise deterministic Markov process. With increasing time $t$ the filter converges against the true parameter exponentially fast  (see e.g.\ \cite{GGVdV}).  Let $n\in\N_0$ and assume $p_{T_n}=p$. Then the evolution of $\stpr{p}$ up to the next jump time $T_{n+1}$ is the solution, denoted by $\phi(t)=(\phi_j(t))_{j=1,\ldots,m}$, of the following system of ordinary differential equations
\begin{equation}\label{eqDir:ODEphi}
\dot{\phi}_j = \phi_j\bigg(\sum_{k=1}^m \lambda_k\,\phi_k-\lambda_j\bigg),\quad j=1,\ldots,m, \quad \phi(0) = p \in\Delta_m 
\end{equation}
 and the new state of the filter $p$ at the jump times $\seq{T}$ is 
\begin{equation*}
p_{T_n} = J\big(p_{T_n-}\big),\quad n\in\N,
\end{equation*}
where
\begin{equation}\label{eqDir:J}
J\big(p\big) \defeq \left(\frac{\lambda_1\, p_1}{\sum_{k=1}^m \lambda_k\,p_k},\ldots,\frac{\lambda_m\, p_m}{\sum_{k=1}^m \lambda_k\,p_k}\right)  
\end{equation}
for $p=(p_1,\ldots,p_m)\in\Delta_m$.

\begin{proposition}\label{prDir:GintkernelPsi}
	The $\fG$-intensity kernel of $\Psi=\tseq{T}{Y}{Z}$, denoted by  $\hat{\nu}(t,{d(y,z)})$, is given by
	\begin{equation*}
	\hat{\nu}(t,(A,B)) = \hat{\Lambda}_{t-}\,F(A)\,\DBsum\fracbetaqt,\quad t\ge0,\;A\in\mathcal{B}((0,\infty)^d),\; B\in\PPD.
	\end{equation*}
	where $\|\cdot\|$ is the $\ell_1$-norm.
\end{proposition}

\begin{proof}
First note that $\hat{\nu}$ is a transition kernel. The  $\fG$-intensity is derived from the $\fG \vee \sigma(\bar{\alpha},\Lambda)-$ intensity $\Lambda F(A) \sum_{D\in B}\alpha_D$ by conditioning on $\cG_t$. Note here in particular that $\Eop[\alpha_D | \cG_t]= \fracbetaqt$ (see Corollary \ref{co:Dirconjugated}) and $\Eop[\Lambda |\cG_t]= \hat{\Lambda}_t$.
\end{proof}

We denote by $\hat{\Psi}(dt, d(y,z))$ the compensated random measure given by
\begin{equation}\label{notDir:Psihat}
\hat{\Psi}(dt, d(y,z)) :=  \Psi(dt, d(y,z)) - \hat\nu(t,d(y,z))dt,
\end{equation}
where $\hat\nu$ is defined as in Proposition~\ref{prDir:GintkernelPsi}.
Thus, we obtain the following indistinguishable representation of the surplus process $\Xxib$:
\begin{equation}\label{eq:wealth}
\begin{aligned}
\dif \Xxib_t = \bigg(&r X_s^{\xi,b} + (\mu-r)\xi_s + c(b_s) - \hat{\Lambda}_t\,b_t\Dsum\fracbetaqt\idsum\ind{D}(i)\,\Eop[Y_1^i]\bigg)\dif t  \\
& + \xi_s\sigma\dif W_s - \int_E b_t\sum_{i=1}^d y_i \ind{z}(i)\,\hat{\Psi}(\dtyz),\quad t\ge0.
\end{aligned}
\end{equation}

This dynamic will be one part of the reduced control model discussed in the next section. 

\subsection{The Reduced Control Problem}
The processes $(p_t)_{t\ge0}$ in \eqref{eq:filterdef} and $(q_t)_{t\ge0}$ in \eqref{eq:qD} carry all relevant information about the unknown parameters $\lambda$ and $\alphabar$ contained in the observable filtration $\fG$ of the insurer.  Therefore, the state process of the reduced control problem with complete observation is the $(\ell+m+1)$-dimensional process
$$(\Xxib_s,p_s,q_s)_{s\in[t,T]},$$ 
where $(\Xxib_s)$ is given by \eqref{eq:wealth}, $(p_s)$ is given by \eqref{eqDir:pj} and $(q_s)$ is given by \eqref{eq:qD}
for some fixed initial time $t\in[0,T)$ and $\xib\in\UtT$. 
We can now formulate the reduced control problem. For any $(t,x,p,q)\in\domainVDir$, the value functions are given by
\begin{equation}\label{eqDir:P} \tag{P}
\begin{aligned}
V^{\xi,b}(t,x,p,q) &\defeq \Eop^{t,x,p,q}\big[U(\Xxib_T)\big],\\
V(t,x,p,q) &\defeq \sup_{(\xi,b)\in\UtT}V^{\xi,b}(t,x,p,q),
\end{aligned}
\end{equation}
where $ \Eop^{t,x,p,q}$ denotes the conditional expectation given $(\Xxib_t,p_t,q_t)=(x,p,q)$. As before, an investment-reinsurance strategy $(\xi^*,b^*)\in\UtT$ is optimal if
$V(t,x,p,q) = V^{\xi^*,b^*}(t,x,p,q).$
%and the insurer is interested in optimal strategies $(\xi^*,b^*)\in\UtT$, i.e.\ in strategies
%\begin{equation*}
%(\xi^*,b^*) = \arg\sup_{\xib\in\UtT} V^{\xi,b}(t,x,p,q).
%\end{equation*}

\section{The Solution}\label{sec:sol}
In a first step we derive the Hamilton-Jacobi-Bellman (HJB) equation. Using standard methods and assuming full differentiability of $V$ we obtain
\begin{equation}\label{eqDir:HJBV}
\begin{aligned}
0&=\sup_{(\xi,b)\in\R\times[0,1]} \bigg\{ V_t(t,x,p,q) - \sum_{k=1}^m \lambda_k\,p_kV(t,x,p,q) + \frac12\sigma^2V_{xx}(t,x,p,q)\xi^2  \\
&\quad+V_x(t,x,p,q)\big(rx + (\mu-r)\xi+c(b)\big) + \sum_{j=1}^m V_{p_j}(t,x,p,q)p_j\bigg(\sum_{k=1}^m \lambda_k\,p_k-\lambda_j\bigg) \\
&\quad+ \sum_{k=1}^m\lambda_k\,p_k\Dsum\fracbetaq\int\limits_{(0,\infty)^d}\!\!\!\! V\Big(t,x-b\idsum y_i\ind{D}(i),J(p),v(q,D)\Big) F(\dif y)\bigg\},
\end{aligned}
\end{equation}
where $v(q,D) := (q_E+\ind{\{E={D}\}})_{E\subset \mathbb{D}}$.
For solving \eqref{eqDir:HJBV} we apply the usual separation approach: For any $(t,x,p,q)\in\domainVDir$, we assume
\begin{equation}\label{eqDir:separation}
V(t,x,p,q) = -e^{-\alpha x e^{r(T-t)}}g(t,p,q).
\end{equation}
This  implies that we conclude from~\eqref{eqDir:HJBV}
\begin{equation}\label{eqDir:HJBgdiff}
\begin{aligned}
&0=\inf_{(\xi,b)\in\R\times[0,1]} \bigg\{ g_t(t,p,q) - \sum_{k=1}^m \lambda_k\,p_k\,g(t,p,q) +\sum_{j=1}^m g_{p_j}(t,p,q)p_j\bigg(\sum_{k=1}^m \lambda_k\,p_k-\lambda_j\bigg) \\
& -\alpha\,e^{r(T-t)}g(t,p,q)\Big((\mu-r)\xi + c(b) - \frac12\alpha\,\sigma^2\,e^{r(T-t)}\xi^2\Big)  \\
&+\sum_{k=1}^m\lambda_k\,p_k\Dsum\fracbetaq g(t,J(p),v(q,D)) \!\!\int\limits_{(0,\infty)^d}\!\!\!\! \exp\bigg\{\alpha\,b\,e^{r(T-t)}\idsum y_i\ind{D}(i)\bigg\} F(\dif y) \bigg\}.
\end{aligned}
\end{equation}
However, $V$ is probably not differentiable w.r.t.\ $t$ and $p_j$, $j=1,\ldots,m$. Assuming $(t,p)\mapsto g(t,p,q)$ is Lipschitz on $[0,T]\times\Delta_m$ for all $q\in\N_0^\ell$, we can replace the partial derivatives of $g$ w.r.t.\ $t$ and $p_j$, $j=1,\ldots,m$, by the generalized Clark gradient (see appendix).
 Throughout, we denote by ${\mathcal L}$ an operator acting on functions $g:\domaingDir\to(0,\infty)$ and $\xib\in\R\times[0,1]$ which is defined by
\begin{equation}\label{eqDir:L}
\begin{aligned}
&{\mathcal L} g(t,p,q;\xi,b):=  - \sum_{k=1}^m \lambda_k\,p_k\,g(t,p,q)+ \alpha\,e^{r(T-t)}g(t,p,q) f_1(t,\xi) +  f_2(t,p,q,b), %\\
%&\defeq - \sum_{k=1}^m \lambda_k\,p_k\,g(t,p,q) -\alpha\,e^{r(T-t)}g(t,p,q)\Big((\mu-r)\xi + c(b) - %\frac12\alpha\,\sigma^2\,e^{r(T-t)}\xi^2\Big) \\
%& +\sum_{k=1}^m\lambda_k\,p_k\Dsum\fracbetaq g\big(t,J(p),v(q,D)\big) \!\!\int\limits_{(0,\infty)^d}\!\!\!\! \exp\bigg\{\alpha\,b\,e^{r(T-t)}\idsum y_i\ind{D}(i)\bigg\} F(\dif y).
\end{aligned}
\end{equation}
where
\begin{equation}\label{eq:fone}
f_1(t,\xi) := -(\mu-r)\xi + \frac12\sigma^2\,\alpha\,e^{r(T-t)}\xi^2
\end{equation}
and
\begin{align*}
& f_2(t,p,q,b) := -\alpha\,e^{r(T-t)}\,g(t,p,q)(\eta-\theta)\,\kappa -\alpha\,e^{r(T-t)}\,g(t,p,q)(1+\theta)\,\kappa\,b \\
&+\sum_{k=1}^m\lambda_k\,p_k\Dsum\fracbetaq g\big(t,J(p),v(q,D)\big) \intdalt\exp\bigg\{\alpha\,b\,e^{r(T-t)}\idsum y_i\ind{D}(i)\bigg\}\! F(\dif y).
\end{align*}

Using this operator and replacing the partial derivatives of $g$ w.r.t.\ $t$ and $p_j$, $j=1,\ldots,m$,  in~\eqref{eqDir:HJBgdiff} by the generalized Clarke gradient, we get the generalized HJB equation for $g$:
\begin{equation}\label{eqDir:HJBg}
0 = \inf_{(\xi,b)\in\R\times[0,1]}\big\{ {\mathcal L} g(t,p,q;\xi,b)\big\} + \inf_{\varphi\in\partial^C\! g_q(t,p)}\bigg\{\varphi_0 + \sum_{j=1}^m\varphi_j\,p_j\bigg(\sum_{k=1}^m \lambda_k\,p_k-\lambda_j\bigg)\bigg\}
\end{equation}
for all $(t,p,q)\in\domaingDir$ with boundary condition
\begin{equation}\label{eqDir:HJBgbound}
g(T,p,q) = 1,\quad (p,q)\in\Delta_m\times\N_0^\ell.
\end{equation}
Note that  we set $\partial^C\! g_q(t,p)=\{\nabla g_q(t,p)\}$ at the points $(t,p)$ where the gradient exists.
The notation $g_q(t,p)$ indicates that the derivative is w.r.t.\ $t$ and $p$ for fixed $q$.

\subsection{Candidate for an Optimal Strategy}
To obtain candidates for an optimal strategy, we rewrite the generalized HJB equation~\eqref{eqDir:HJBg} as
\begin{equation}\label{eqDir:gHJBa}
\begin{aligned}
0 &= - \sum_{k=1}^m\lambda_k\,p_k\,g(t,p,q) + \alpha\,e^{r(T-t)}g(t,p,q)\inf_{\xi\in\R} f_1(t,\xi) + \inf_{b\in[0,1]} f_2(t,p,q,b) \\
&\quad + \inf_{\varphi\in\partial^C g_q(t,p)}\bigg\{\varphi_0 + \sum_{j=1}^m\varphi_j\,p_j\bigg(\sum_{k=1}^m \lambda_k\,p_k-\lambda_j\bigg)\bigg\}.
\end{aligned}
\end{equation}
%where 
%\begin{equation}\label{eq:fone}
%f_1(t,\xi) \defeq -(\mu-r)\xi + \frac12\sigma^2\,\alpha\,e^{r(T-t)}\xi^2
%\end{equation}
%and
%\begin{align*}
%& f_2(t,p,q,b) = -\alpha\,e^{r(T-t)}\,g(t,p,q)(\eta-\theta)\,\kappa -\alpha\,e^{r(T-t)}\,g(t,p,q)(1+\theta)\,\kappa\,b \\
%&+\sum_{k=1}^m\lambda_k\,p_k\Dsum\fracbetaq g\big(t,J(p),v(q,D)\big)\times \intdalt\exp\bigg\{\alpha\,b\,e^{r(T-t)}\idsum %y_i\ind{D}(i)\bigg\}\! F(\dif y).
%\end{align*}
Hence we can conclude  that the unique candidate of an optimal investment strategy $\prTalt{\xi^{\star}}$ is given by 
\begin{equation}\label{eqDir:xistarfunc}
\xi^{\star}(t) = \frac{\mu-r}{\sigma^2} \frac1\alpha e^{-r(T-t)},\quad t\in[0,T].
\end{equation}

The following lemma will yield the first order condition for a candidate of an optimal reinsurance strategy. In order to avoid confusion with a reinsurance strategy, we will use $a\in\R$ instead of $b=(b_t)$ as an argument for the function $f_2$ in the following.

\begin{lemma}\label{leDir:ftwo}
	For any $(t,p,q)\in\domaingDir$, the function $\R\ni a\mapsto f_2(t,p,q,a)$ is strictly convex and
	\begin{align*}
	\partialvar{a}f_2(t,p,q,a) &= -\alpha\,e^{r(T-t)}\Bigg(g(t,p,q)\,(1+\theta)\,\kappa -\Dsum\fracbetaq g\big(t,J(p),v(q,D)\big)\times \\
	&\quad \idsum\ind{D}(i)\intdalt y_i\,\exp\bigg\{\alpha\,a\,e^{r(T-t)}\sum_{j=1}^d y_j\ind{D}(j)\bigg\} F(\dif y) \sum_{k=1}^m\lambda_k\,p_k\Bigg).
	\end{align*}
\end{lemma}

\begin{proof}
	Strict convexity follows since $f_2$ is the sum of a linear and a strictly convex function in $a$. The derivative is straightforward.
\end{proof}

  For any $(t,p,q)\in\domaingDir$ and $a\in\R$, we define in case $g>0$
\begin{equation}\label{eqDir:hb}
\begin{aligned}
h(t,p,q,a)\defeq &\sum_{k=1}^m\lambda_k\, p_k \Dsum\fracbetaq  \frac{g(t,J(p),v(q,D))}{g(t,p,q)}\times \\
&\idsum\ind{D}(i)\intdalt y_i\exp\bigg\{\alpha\,a\,e^{r(T-t)} \sum_{j=1}^d y_j\ind{D}(j)\bigg\} F(\dif y).
\end{aligned}
\end{equation}
Furthermore, we set
\begin{align*}
A(t,p,q) &\defeq h(t,p,q,0), \\
B(t,p,q) &\defeq h(t,p,q,1).
\end{align*}

Obviously $A(t,p,q) \le B(t,p,q)$. Setting $\partialvar{b}f_2$ to zero (c.f.\ Lemma~\ref{leDir:ftwo}), we obtain the first order condition
\begin{equation}\label{eqDir:ftwofoc}
(1+\theta)\,\kappa = h(t,p,q,a).
\end{equation}
 If  a minimizer exists it is unique due to the strict convexity property of $f_2$ w.r.t.\ $a$.
The next proposition states that this equation is solvable and the solution takes values in $[0,1]$ depending on the safety loading parameter $\theta$ of the reinsurer. 

\begin{proposition}\label{prDir:candidateb}
	For any $(t,p,q)\in\domaingDir$, Equation~\eqref{eqDir:ftwofoc} has a unique root w.r.t.\ $a$, denoted by $r(t,p,q)$, which is increasing w.r.t.\ the safety loading  parameter of the reinsurer $\theta$. Moreover,
	it holds, 
	\begin{enumerate}
		\item $r(t,p,q) \le 0$ if $\theta\le A(t,p,q)/\kappa-1$,
		\item $0< r(t,p,q) <1$ if $A(t,p,q)/\kappa-1< \theta < B(t,p,q)/\kappa-1$,
		\item $r(t,p,q)\ge 1$ if $\theta\ge B(t,p,q)/\kappa-1$.
	\end{enumerate}
\end{proposition}

\begin{proof}
Note that $a\mapsto h(t,p,q,a)$ is strictly increasing. The proof then follows from considering the zeros of \eqref{eqDir:ftwofoc} in $\theta$ when $a=0$ and when $a=1$.
\end{proof}

Therefore, the proposition above provides the candidate for an optimal reinsurance strategy.
For any $(t,p,q)\in\domaingDir$, we set
\begin{equation}\label{eqDir:bstarfunc}
b(t,p,q) \defeq \begin{cases} 0, & \theta\le A(t,p,q)/\kappa-1,\\
1, &\theta\ge B(t,p,q)/\kappa-1, \\
r(t,p,q), &\text{otherwise}.
\end{cases}
\end{equation}
Then the candidate for an optimal reinsurance strategy  is given by $b^{\star}(t) \defeq b(t-,p_{t-},q_{t-})$.

\subsection{Verification}
This section is devoted to a verification theorem to ensure that the solution of the stated generalized HJB equation yields the value function (see Theorem~\ref{thDir:veri}). We also demonstrate an existence theorem of a solution of the HJB equation (see Theorem~\ref{thDir:existanceHJB}).  Both proofs can be found in the appendix.

\begin{theorem}\label{thDir:veri}
	Suppose there exists a bounded function $h:\domaingDir\to(0,\infty)$  such that $t\mapsto h(t,p,q)$ and $t\mapsto h(t,\phi(t,p),q)$ are Lipschitz on $[0,T]$ for all $(p,q)\in\Delta_m\times\N_0^\ell$ as well as $p\mapsto h(t,p,q)$ is concave for all $(t,q)\in[0,T]\times\N_0^\ell$. Furthermore, $h$ satisfies the generalized HJB equation \eqref{eqDir:HJBg}
%	\begin{equation}\label{eqDir:hHJB}
%	0=\inf_{(\xi,b)\in\R\times[0,1]}\{\mathcal{L} h(t,p,q,\xi,b)\} + \inf_{\varphi\in\partial^C h_{q}(t,p)}\bigg\{\varphi_0 + \sum_{j=1}^m\varphi_j\,p_j\bigg(\sum_{k=1}^m \lambda_k\,p_k-\lambda_j\bigg)\bigg\},
%	\end{equation}
	for all $(t,p,q)\in[0,T)\times\Delta_m\times\N_0^\ell$ with boundary condition
	\begin{equation}\label{eqDir:hHJBbcond}
	h(T,p,q) = 1,\quad (p,q)\in\Delta_m\times[0,T].
	\end{equation}
	Then
	\begin{equation*}
	V(t,x,p,q)= -e^{-\alpha x e^{r(T-t)}}h(t,p,q),\quad (t,x,p,q)\in\domainVDir,
	\end{equation*}
	and $\dprtTalt{\xi^\star}{b^\star}$ with $\xi^\star(s)$ given by~\eqref{eqDir:xistarfunc} and $b^\star(s)\defeq b(s-,p_{s-},q_{s-})$ given by~\eqref{eqDir:bstarfunc} (with $g$ replaced by $h$ in $A(s,p,q)$ and $B(s,p,q)$) 
	is an optimal feedback strategy for the given optimization problem~\eqref{eqDir:P}, i.e.\ $V(t,x,p,q) = V^{\xi^\star,b^\star}(t,x,p,q)$.
\end{theorem}

\subsection{Existence result for the value function}
\label{secDir:existencevalue}
%==========================================================================================

We now show that there exists a function $h:\domaingDir\to(0,\infty)$ satisfying the conditions stated in Theorem~\ref{thDir:veri}. For this purpose let 
\begin{equation}\label{eqDir:g}
g(t,p,q) := \inf_{(\xi,b)\in\UtT}g^{\xi,b}(t,p,q),
\end{equation}
where
\begin{equation}\label{eqDir:gxib}
\begin{aligned}
&g^{\xi,b}(t,p,q) \defeq \Eop^{t,p,q}\bigg[\exp\bigg\{-\int_t^T \alpha\, e^{r(T-s)}\big((\mu-r)\,\xi_s+c(b_s)\big)ds\\
&\quad -\int_t^T\alpha\,\sigma\, e^{r(T-s)}\xi_sdW_s  +\int_t^T \intEd \alpha\, b_s\, e^{r(T-s)} \idsum y_i\ind{z}(i)\,\Psi(\dsyz)\bigg\}\bigg],
\end{aligned}
\end{equation}
where $\Eop^{t,p,q}$ denotes the conditional expectation given $(p_t,q_t)=(p,q)$. The next lemma summarizes useful properties of $g$. A proof can be found in the appendix.

\begin{lemma}\label{leDir:propg}
	The function $g$ defined by~\eqref{eqDir:g} has the following properties:
	\begin{enumerate}
		\item $g$ is bounded on $\domaingDir$ by a constant $0<K_1<\infty$ and $g>0$.
		\item $g^{\xi,b}(t,p,q) = \sum_{j=1}^m p_j\,g^{\xi,b}(t,e_j,q)$ for all $(t,p,q)\in\domaingDir$ and $\xib\in\UtT$.
		 \item $\Delta_m\ni p\mapsto g(t,p,q)$ is concave for all $(t,q)\in[0,T]\times\N_0^\ell$.
		\item $[0,T]\ni t\mapsto g(t,p,q)$ is Lipschitz on $[0,T]$ for all $(p,q)\in\Delta_m\times\N_0^\ell$.
		\item $[0,T]\ni t\mapsto g(t,\phi(t),q)$ with $\phi(0)=p$ is Lipschitz on $[0,T]$ for all $(p,q)\in\Delta_m\times\N_0^\ell$.
%		 \item Let $M$ be the set of all points $(t,p,q)\in\domaingDir$, where $D\,g$ exists. Then there exists a constant $0<K_4<\infty$ such that $|D\,g(t,p,q)|\le K_4$ for all $(t,p,q)\in M$.
%		\item There exists a constant $0<K_5<\infty$ such that $\big|\mathcal{L}g(t,p,q;\xi,b)\big|\le K_5$ for all $(t,p,q)\in\domaingDir$ and $\xib\in[-K,K]\times[0,1]$. 
%		\item There exists a constant $0<K_6<\infty$ such that $\big|\inf_{(\xi,b)\in[-K,K]\times[0,1]}\mathcal{L}g(t,p,q;\xi,b)\big|\le K_6$ for all $(t,p,q)\in\domaingDir$.	
	%	\item $g^{\xi,b}(t,p,q) = \sum_{j=1}^m p_j\,g^{\xi,b}(t,e_j,q)$ for all $(t,p,q)\in\domaingDir$ and $\xib\in\UtT$.
	%	\item $g^{\xi,b}(t,J(p),q)=\sum_{j=1}^m \frac{\lambda_j\,p_j}{\sum_{k=1}^m\lambda_k\,p_k}g^{\xi,b}(t,e_j,q)$ for all $(t,p,q)\in\domaingDir$ and $\xib\in\UtT$.
	%	\item $g^{\xi,b}(t,p,q)=\Dsum\fracbetaq g^{\xi,b}(t,p,v(q,D))$ for all $(t,p,q)\in\domaingDir$ and $\xib\in\UtT$.
	%	\item $\Delta_m\ni p\mapsto g(t,p,q)$ is concave for all $(t,q)\in[0,T]\times\N_0^\ell$.   
	\end{enumerate}
\end{lemma}

Notice that $e_j$ denotes the $j$th unit vector. We are now in the position to show the following existence result of a solution of the generalized HJB equation.

\begin{theorem}\label{thDir:existanceHJB}
	The value function of  problem~\eqref{eqDir:P} is given by
	\begin{equation*}
	V(t,x,p,q) = -e^{-\alpha x e^{r(T-t)}}g(t,p,q),\quad (t,x,p,q)\in\domainVDir,
	\end{equation*}
	where $g$ is defined by~\eqref{eqDir:g} and satisfies the generalized HJB equation \eqref{eqDir:HJBg}
%	\begin{equation*}
%	0=\inf_{(\xi,b)\in\R\times[0,1]}\{\mathcal{L} g(t,p,q;\xi,b)\} + \inf_{\varphi\in\partial^C g_q(t,p)}\bigg\{\varphi_0 + \sum_{j=1}^m\varphi_j\,p_j\bigg(\sum_{k=1}^m \lambda_k\,p_k-\lambda_j\bigg)\bigg\}
%	\end{equation*}
	for all $(t,p,q)\in\domaingDir$ with boundary condition $g(T,p,q)=1$ for all $(p,q)\in\Delta_m\times\N_0^\ell$. 
	Furthermore, $\dprtTalt{\xi^{\star}}{b^{\star}}$ with $\xi^{\star}(s)$ given by~\eqref{eqDir:xistarfunc} and $b^{\star}(s)=b(s-,p_{s-},q_{s-})$ given by~\eqref{eqDir:bstarfunc} is the optimal investment and reinsurance strategy of the optimization problem~\eqref{eqDir:P}. 
\end{theorem}

\section{Comparison Results with the complete Information Case }\label{sec:comp}
First note that the case with complete information is always a special case of our general model. We obtain this case when the prior is concentrated on a single value. With complete information the optimal investment strategy is given by 
\begin{equation*}
\xi^{\star}(t) = \frac{\mu-r}{\sigma^2} \frac1\alpha e^{-r(T-t)},\quad t\in[0,T],
\end{equation*}
which is exactly the same as in the case of partial observation. This is no surprise since the partial observation only concerns the reinsurance strategy. In order to state the optimal reinsurance strategy in the complete information case define   for any $t\in[0,T]$ and $a\in\R$
\begin{equation}\label{eqfull:hb}
h_{\lambda, c}(t,a) \defeq \lambda \Dsum c_D \gamma(t,a,D)
\end{equation}
with 
\begin{equation} \gamma(t,a,D) := \idsum\ind{D}(i) \intdalt y_i\exp\bigg\{\alpha\, a\, e^{r(T-t)} \sum_{j=1}^d y_j\ind{D}(j)\bigg\} F(\dif y).
\end{equation}
Furthermore, we define
\begin{align*}
A_{\lambda, c}(t) &\defeq h_{\lambda, c}(t,0), \\
B_{\lambda, c}(t) &\defeq h_{\lambda, c}(t,1).
\end{align*}
 From now on, $a_{\lambda, c}(t)$ denotes the unique root of
 \begin{equation}\label{eqfull:ftwofoc}
 (1+\theta)\kappa = h_{\lambda, c}(t,b)
 \end{equation}
which exists. By the same line of arguments as in Proposition~\ref{prDir:candidateb}, we obtain under the notation above that the optimal reinsurance strategy $b^\star_{\lambda, c}$ is given by 
\begin{equation}\label{eqfull:bstart}
b^\star_{\lambda, c}(t)\defeq 
\begin{cases} 0, & 0\le A_{\lambda, c}(t)/\kappa -1, \\
1, &\theta \ge B_{\lambda, c}(t)/\kappa-1, \\ 
a_{\lambda, c}(t), &\text{otherwise}.
\end{cases}
\end{equation}
Note that $a_{\lambda, c}(t)$, $A_{\lambda, c}(t)$ and $B_{\lambda, c}(t)$ are continuous in $t$. Consequently, the optimal reinsurance strategy $b^\star_{\lambda, c}$ is continuous. Moreover, $b^\star_{\lambda, c}$ is deterministic and can be calculated easily.

We will now compare the reinsurance strategies. In order to do so, we need the following properties of $g$ (see appendix for the proof):

\begin{lemma}\label{leDir:propg2}
	The function $g$ defined by~\eqref{eqDir:g} has the following properties for $\xib\in\UtT$:
	\begin{enumerate}
	%	\item $g$ is bounded on $\domaingDir$ by a constant $K_3$.
	%		\item $\Delta_m\ni p\mapsto g(t,p,q)$ is concave for all $(t,q)\in[0,T]\times\N_0^\ell$.
	%	\item $[0,T]\ni t\mapsto g(t,p,q)$ is Lipschitz on $[0,T]$ for all $(p,q)\in\Delta_m\times\N_0^\ell$.
	%	\item $[0,T]\ni t\mapsto g(t,\phi(t,p),q)$ is Lipschitz on $[0,T]$ for all $(p,q)\in\Delta_m\times\N_0^\ell$.
	%	\item Let $M$ be the set of all points $(t,p,q)\in\domaingDir$, where $D\,g$ exists. Then there exists a constant $0<K_4<\infty$ such that $|D\,g(t,p,q)|\le K_4$ for all $(t,p,q)\in M$.
	%	\item There exists a constant $0<K_5<\infty$ such that $\big|\mathcal{L}g(t,p,q;\xi,b)\big|\le K_5$ for all $(t,p,q)\in\domaingDir$ and $\xib\in[-K,K]\times[0,1]$. 
	%	\item There exists a constant $0<K_6<\infty$ such that $\big|\inf_{(\xi,b)\in[-K,K]\times[0,1]}\mathcal{L}g(t,p,q;\xi,b)\big|\le K_6$ for all $(t,p,q)\in\domaingDir$.	
	%	\item $g^{\xi,b}(t,p,q) = \sum_{j=1}^m p_j\,g^{\xi,b}(t,e_j,q)$ for all $(t,p,q)\in\domaingDir$ and $\xib\in\UtT$.
		\item $g^{\xi,b}(t,J(p),q)=\sum_{j=1}^m \frac{\lambda_j\,p_j}{\sum_{k=1}^m\lambda_k\,p_k}g^{\xi,b}(t,e_j,q)$ for all $(t,p,q)\in\domaingDir$.% and $\xib\in\UtT$.
		\item $g^{\xi,b}(t,p,q)=\Dsum\fracbetaq g^{\xi,b}(t,p,v(q,D))$ for all $(t,p,q)\in\domaingDir$.% and $\xib\in\UtT$.
		%	\item $\Delta_m\ni p\mapsto g(t,p,q)$ is concave for all $(t,q)\in[0,T]\times\N_0^\ell$.   
	\end{enumerate}
\end{lemma}

First of all we derive bounds for the optimal strategy which can be calculated a priori, i.e. independent of the filter process $\stpr{p}$ and the process $\stpr{q}$. For this determination, we introduce the following terms. Let $t\in[0,T]$ and $a\in\R$. Throughout this section, we set
	\begin{align*}
	h^{\min}(t,a) &\defeq \lambda_1\min_\DD\big\{\gamma(t,a,D)\big\}, \quad h^{\max}(t,a) \defeq \lambda_m\max_\DD\big\{\gamma(t,a,D)\big\}.
	\end{align*}

The proof of the next result is straightforward:

\begin{proposition}\label{prDir:prophminmax}
	Let $t\in[0,T]$. Then $\R\ni a\mapsto h^{\min}(t,a)$ and $\R\ni a\mapsto h^{\max}(t,a)$ are strictly increasing and strictly convex. Furthermore, 
	\begin{equation*}
	\lim_{a\to-\infty}h^{\min}(t,a) = \lim_{a\to-\infty}h^{\max}(t,a)= 0,\quad \lim_{a\to\infty} h^{\min}(t,a) = \lim_{a\to\infty} h^{\max}(t,a)= \infty. 
	\end{equation*}
\end{proposition}
This proposition justifies the following notation:
For some fixed $t\in[0,T]$, we denote by $a^{\min}(t)$ the unique root of the equation $(1+\theta)\,\kappa = h^{\min}(t,a)$ w.r.t.\ $a$, and by $a^{\max}(t)$ the unique root of the equation $(1+\theta)\,\kappa = h^{\max}(t,a)$ w.r.t.\ $a$.
The announced a-priori-bounds are a direct consequence of the following theorem in connection with Proposition~\ref{prDir:prophminmax}.

\begin{proposition}\label{prDir:bounds}
	For any $(t,p,q)\in\domaingDir$ and $a\in\R$, we have for $h$ from \eqref{eqDir:hb}
	\begin{equation*}
	h^{\min}(t,a) \le h(t,p,q,a) \le h^{\max}(t,a).
	\end{equation*}
\end{proposition}

\begin{proof}
	Choose some $(t,p,q)\in\domaingDir$ and $a\in\R$. Recall that $\lambda_1<\lambda_2<\ldots<\lambda_m$. For any $\xib\in\UtT$, an application of Lemma~\ref{leDir:propg2} yields
	\begin{align*}
	&\quad\sum_{k=1}^m\lambda_k p_k \Dsum\fracbetaq g^{\xi,b}(t,J(p),v(q,D)) \gamma(t,a,D)\\	&=\Dsum\fracbetaq\sum_{j=1}^m\lambda_j\, p_j\,g^{\xi,b}(t,e_j,v(q,D))\gamma(t,a,D)\\
%&\le \hmaxF(t,a)\Dsum\fracbetaq \sum_{j=1}^m p_j g^{\xi,b}(t,e_j,v(q,D)) \\
	&\le h^{\max}(t,a)\Dsum\fracbetaq g^{\xi,b}(t,p,v(q,D)) =h^{\max}(t,a)\,g^{\xi,b}(t,p,q).
	\end{align*}
	Hence, by taking the infimum over all $\xib\in\UtT$ on both sides, we get $h(t,p,q,a) \le h^{\max}(t,a)$. The other announced inequality is obtained in the same way.
\end{proof}

The proposition directly implies the following corollary:

\begin{corollary}\label{coDir:bounds}
	The optimal reinsurance strategy $b^{\star}$ from Theorem~\ref{thDir:existanceHJB} has the following bounds:
	\begin{equation*}
	\max\{0,a^{\max}(t)\} \le b^{\star}(t) \le \min\{1,a^{\min}(t)\},\quad t\in[0,T].
	\end{equation*}
\end{corollary}
These bounds provide only a rough estimate for the optimal reinsurance strategy.
The next theorem provides the comparison statement. For this theorem we need the following assumption:
	From now on, we suppose that 
	\begin{equation} \label{asDir:comparison}
	F(\dif y) = \Fbar(\dif y_1)\otimes\Fbar(\dif y_2)\otimes\cdots\otimes\Fbar(\dif y_m),
	\end{equation}
	where $\Fbar$ is a distribution on $(0,\infty)$ with existing moment generation function.

The next theorem is now the main statement of this section. It provides a comparison of the optimal reinsurance strategy to the optimal one in the case of complete information where the unknown quantities $\Lambda$ and $\bar{\alpha}$ are replaced by their expectations.

\begin{theorem}\label{thDir:comparison}
	Let  \eqref{asDir:comparison} be fullfilled, $b$ be the function given by~\eqref{eqDir:bstarfunc} and $b^\star_{\lambda,c}$ the function given by~\eqref{eqfull:bstart}. Then, for any $(t,p,q)\in\domaingDir$,
	\begin{equation*}
	b(t,p,q) \le b^\star_{u(p),w(q)}(t)
	\end{equation*}
	with
	\begin{equation*}
	u(p)\defeq \sum_{k=1}^m\lambda_k p_k,\quad w(q)\defeq\bigg(\fracbetaq\bigg)_\DD.
	\end{equation*}
\end{theorem}

\begin{proof}
 For any $q=\Dseq{q}\in\N_0^\ell$, we define $\qtilde=(\qtilde_1,\ldots,\qtilde_d)\in\N_0^d$ by
\begin{equation*}
\qtilde_i \defeq \sum_{\substack{\DD: \\ |D|=i}}q_D
\end{equation*}	
i.e.\ whereas the components of $q$ count the number of events where claims in set $D$ are affected, the components of $\qtilde$ count the number of events where $i$ lines are affected. 
Due to Assumption~\eqref{asDir:comparison} $\qtilde$ contains the same information as $q$. Thus, we can interpret $g^{\xi,b}(t,p,q)$ as a function $g^{\xi,b}(t,p,\qtilde).$ With a slight abuse of notation we keep the same name for the function.  We also have that $i\mapsto g^{\xi,b}(t,p,\qtilde+e_i)$ is increasing. We obtain from Lemma ~\ref{leDir:propg2} a) 
\begin{align*} 
&\sum_{k=1}^m\lambda_k p_k \Dsum\!\fracbetaq g^{\xi,b}(t,J(p),v(q,D))\gamma(t,a,D)\\ =& \sum_{j=1}^m p_j\lambda_j\Dsum\fracbetaq g^{\xi,b}(t,e_j,v(q,D))\gamma(t,a,D)
\end{align*}
Next we can write this as 	
\begin{align*} 
&=\sum_{j=1}^m p_j\lambda_j\idsum\frac{\tilde{\beta}_i+\qtilde_i}{\normbetaq}\, g^{\xi,b}(t,e_j,\qtilde+e_i)\,i\bigg(\int_{(0,\infty)} e^{\alpha a e^{r(T-t)} y_1}\Fbar(\dif y_1)\bigg)^{i-1}\times\\
&\qquad \int_{(0,\infty)} y_1e^{\alpha a e^{r(T-t)} y_1}\Fbar(\dif y_1)\\
&\ge \sum_{j=1}^m p_j\lambda_j\idsum\frac{\tilde{\beta}_i+\qtilde_i}{\normbetaq}\, g^{\xi,b}(t,e_j,\qtilde+e_i)\sum_{k=1}^d\frac{\tilde{\beta}_k+\qtilde_k}{\normbetaq}\,k\bigg(\int_{(0,\infty)} e^{\alpha a e^{r(T-t)} y_1}\Fbar(\dif y_1)\bigg)^{k-1}\times\\
&\qquad\int_{(0,\infty)} y_1e^{\alpha b e^{r(T-t)} y_1}\Fbar(\dif y_1)\\
%&= \sum_{j=1}^m p_j\lambda_j\Esum\frac{\beta_E+q_E}{\|\bar{\beta}+q\|}g^{\xi,b}(t,e_j,v(q,E))\Dsum\fracbetaq\times \\
%&\quad\idsum\ind{D}(i)\intdalt y_i\exp\bigg\{\alpha\,a\,e^{r(T-t)}\sum_{j=1}^m y_j\ind{D}(j)\bigg\}F(\dif y) \\
%&=\sum_{j=1}^m p_j\lambda_j g^{\xi,b}(t,e_j,q)\Dsum\fracbetaq\idsum\ind{D}(i)\!\!\int\limits_{(0,\infty)^d}\!\!\!\! y_i\exp\bigg\{\alpha a e^{r(T-t)}\sum_{j=1}^m y_j\ind{D}(j)\bigg\}F(\dif y),
\end{align*}	
The inequality is due to Lemma ~\ref{le:ineqsum} and the fact that 	$g^{\xi,b}(t,e_j,\tilde q + e_i)$ and the second factor both are increasing in $i$. The last expression can due to Lemma~\ref{leDir:propg2} b) be written as
\begin{align*} 
& \sum_{j=1}^m p_j\lambda_j\Esum\frac{\beta_E+q_E}{\|\bar{\beta}+q\|}g^{\xi,b}(t,e_j,v(q,E))\Dsum\fracbetaq\gamma(t,a,D) \\
&=\sum_{j=1}^m p_j\lambda_j g^{\xi,b}(t,e_j,q)\Dsum\fracbetaq \gamma(t,a,D).
\end{align*}
Further we have
\begin{align*} 
& \sum_{j=1}^m p_j\lambda_j g^{\xi,b}(t,e_j,q) \ge  g^{\xi,b}(t,p,q) \sum_{j=1}^m p_j\lambda_j.
\end{align*}
again by Lemma ~\ref{le:ineqsum}, Lemma \ref{leDir:propg2} and the fact that $\lambda_j$ and $g^{\xi,b}(t,e_j,q)$ are increasing. 
Thus, we obtain
 \begin{align*}
&\sum_{j=1}^m p_j\lambda_j g^{\xi,b}(t,e_j,q)\Dsum\fracbetaq \gamma(t,a,D) \ge  g^{\xi,b}(t,p,q) \sum_{j=1}^m p_j\lambda_j\Dsum\fracbetaq\gamma(t,a,D).
\end{align*}	
 In summary, we have
\begin{align*}
&\sum_{k=1}^m\!\!\lambda_k p_k \Dsum\fracbetaq g^{\xi,b}(t,J(p),v(q,D))\gamma(t,a,D)\ge g^{\xi,b}(t,p,q) \sum_{j=1}^m\!\! p_j\lambda_j\Dsum\fracbetaq\gamma(t,a,D),
\end{align*}
which yields $h(t,p,q,a)\ge h_{u(p),w(q)}(t,a)$ by taking the infimum over all $\xib\in\UtT$ on both sides and the proof follows by inspecting the minimum points.
\end{proof}

\section{Numerical Results}
In this section we illustrate some numerical results in the case of two lines of business (i.e.\ $d=2$). The set of possible background intensities $\Lambda$ is $\{2,4,5\}$  and the prior probability mass function of $\Lambda$ is supposed to be
\begin{equation*}
\bar{\pi}_\Lambda=\bigg(\frac{2}{5},\frac{2}{5},\frac{1}{5}\bigg).
\end{equation*}
Furthermore, we assume that the prior parameter of the Dirichlet distribution of the thinning probabilities $\alphabar$ is
\begin{equation*}
\bar{\beta} = (8,7,5).
\end{equation*}
 Since we want to present the comparison result graphically, we choose the same claim size distribution for both business lines, namely a right-truncated exponential distribution with rate 1 and truncation at 3, i.e. 
\begin{equation*}
\Eop[Y_1^1]=\Eop[Y_1^2] = \frac{1}{1-e^{-3}}.
\end{equation*}
For the parameter $\kappa$ of the premium principle, we choose
\begin{align*}
\kappa &= \sum_{k=1}^m\lambda_k\,\pi_\Lambda(k)\Dsum\fracbetaD\idsum\ind{D}(i)\Eop\big[Y_1^i\big] 
=\frac{17}{4-4e^{-3}}.
\end{align*}

The remaining parameters are chosen as in Table~\ref{tabDir:simpara}.
\begin{table}[hb]
	\centering
	\begin{tabular}{cc}
		\hline
		parameter & value \\ 
		\hline
		$x_0$ & 100\phantom{.01}  \\
		$T$ & \phantom{0}10\phantom{.01}  \\
		$r$ & \phantom{10}0.01 \\
		$\mu$ & \phantom{10}0.2\phantom{1} \\
		$\sigma$ & \phantom{10}3\phantom{.01} \\
		$\alpha$ & \phantom{10}0.2\phantom{1} \\
		$\eta$ & \phantom{10}0.4\phantom{1} \\
		$\theta$ & \phantom{10}0.6\phantom{1} \\
		\hline
	\end{tabular}
	\caption[Simulation parameters]{Simulation parameters.}\label{tabDir:simpara}
\end{table}
The following simulations are generated under the assumption that the realization of $\alphabar$ is $(0.38,0.48,0.14)$ and that the true background intensity is 4 (i.e.\ the realization of $\Lambda$ is 4). Trajectories of the filter can be seen in Figure~\ref{figDir:filter}. It is illustrative to see the fast convergence against the true parameter.
\begin{figure}[htbp]
	\centering
	\includegraphics[width=0.8\textwidth]{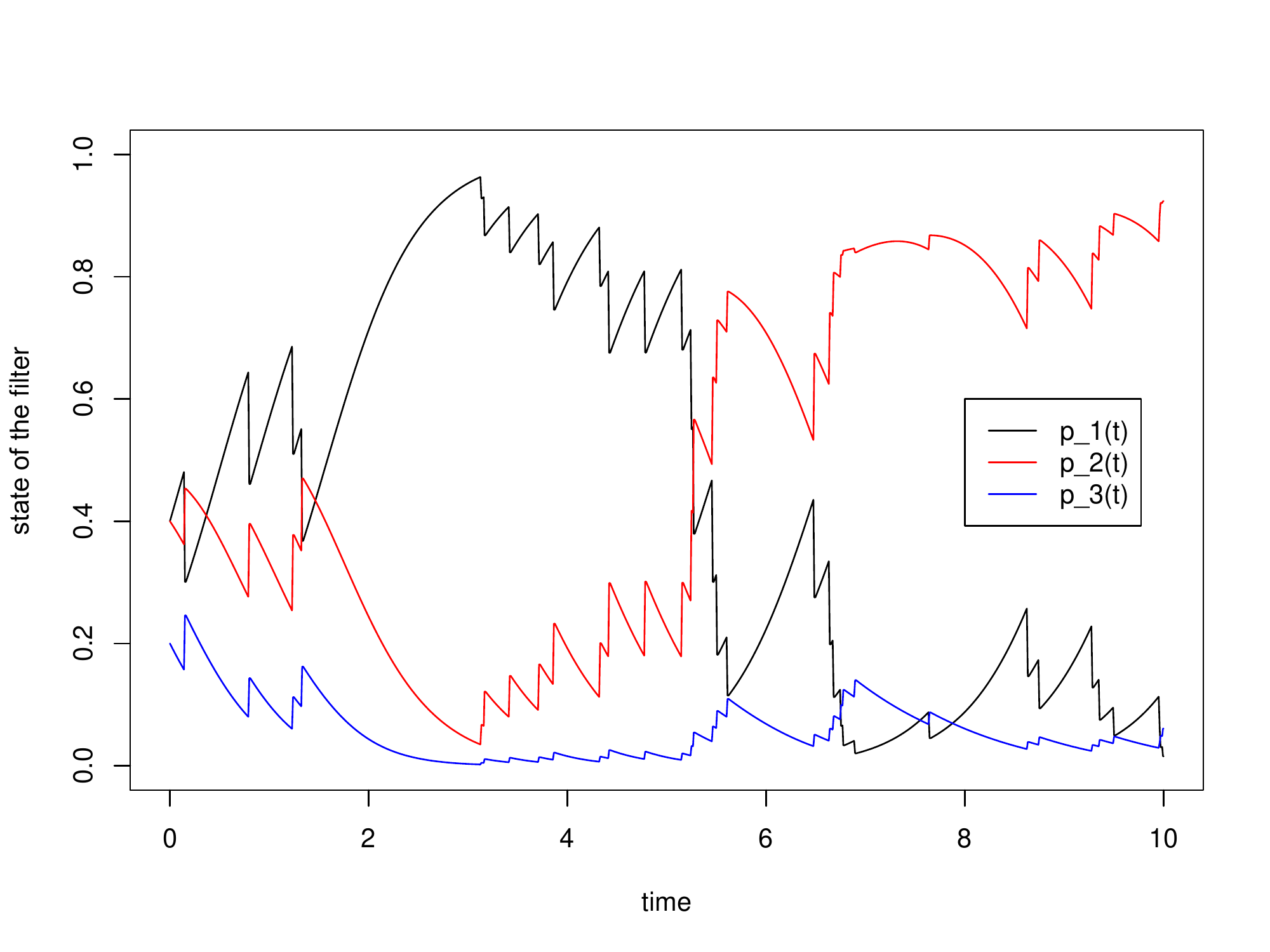}
	\caption[A trajectory of the filter process]{A trajectory of the filter process $\stpr{p}$ under the assumptions that  $\bar{\pi}_\Lambda=(2/5,2/5,1/5)$ and $\Lambda=4$, where $p_t=(p_1(t),p_2(t),p_3(t))$ with $p_1(t)=\P(\Lambda=2|\cG_t)$, $p_2(t)=\P(\Lambda=4|\cG_t)$ and $p_3(t)=\P(\Lambda=5|\cG_t)$.}
	\label{figDir:filter}
\end{figure}

 In Figure~\ref{figDir:bounds}  the a-priori-bounds (red and orange) are illustrated together with two trajectories (black and blue) of the reinsurance strategy $(b^\star_{u(p_{t-}),w(q_{t-})}(t))_{t\in[0,T]}$ with  $u(p)\defeq \sum_{k=1}^m\lambda_k p_k$ and $w(q)\defeq((\beta_D+q_D)/\normbetaq)_\DD$, which provide for each scenario an upper bound for the corresponding optimal reinsurance strategy according to Theorem~\ref{thDir:comparison}. 
So the black and blue lines depend on the realized trigger arrival times and the affected business lines. In both scenarios, the upper bounds (black and blue) obtained from the comparison result are only useful up to approximaltely time 8.
Before this, a strong dependence on the realizations can be seen. Only until the first trigger event both paths provide the same bound.

\begin{figure}[htbp]
	\centering
	\includegraphics[width=0.8\textwidth]{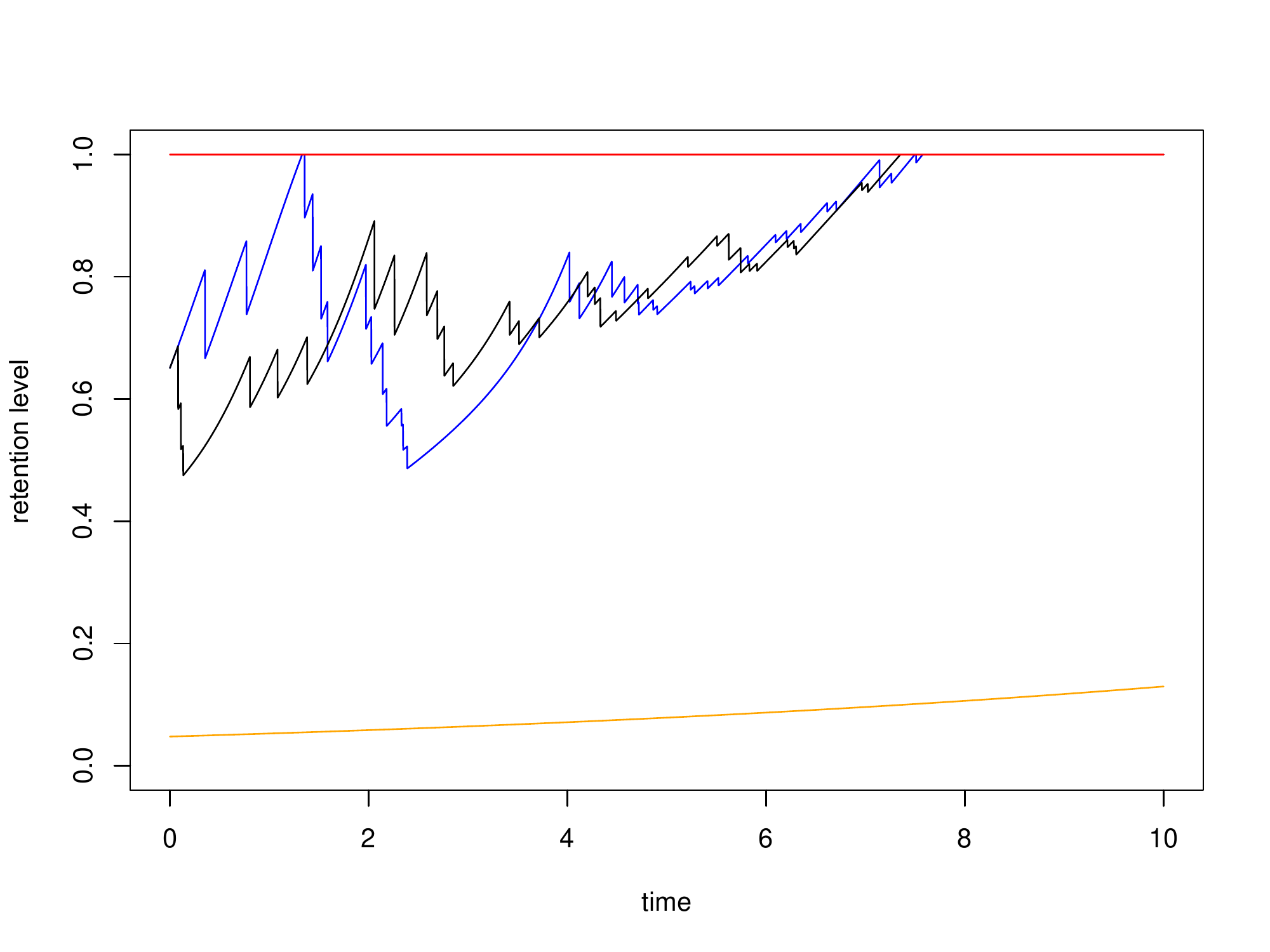}
	\caption[A-priori-bounds and an upper bound for the optimal reinsurance strategy.]{A priori upper (red) and lower bound (orange) for the optimal reinsurance strategy and two paths of the reinsurance strategy $(b^\star_{u(p_{t-}),w(q_{t-})}(t))_{t\in[0,T]}$ with  $u(p)\defeq \sum_{k=1}^m\lambda_k p_k$ and $w(q)\defeq((\beta_D+q_D)/\normbetaq)_\DD$.}
	\label{figDir:bounds}
\end{figure}

Concluding the numerical illustration, we show the path of the surplus process in an insurance loss scenario for three different insurance strategies in Figure~\ref{figDir:surplus}. 
In the case of full reinsurance (i.e. retention level of 0) the trajectory of the surplus process tends downwards (red) due to a negative premium rate.
The blue line displays a trajectory of the surplus  for a constant reinsurance strategy of 0.5 and the black line for the reinsurance strategy $(b^\star_{u(p_{t-}),w(q_{t-})}(t))_{t\in[0,T]}$ with  $u(p)= \sum_{k=1}^m\lambda_k p_k$ and $w(q)=((\beta_D+q_D)/\normbetaq)_\DD$. From Figure~\ref{figDir:bounds}, we known that the latter reinsurance strategy tends upwards, which is evident in Figure~\ref{figDir:surplus} since jump sizes of the black line are higher at the end of the considered time interval than those of the blue line. But because of the lower level of reinsurance, the surplus between losses rises stronger (as the premium rate is higher) than in the case of the constant reinsurance strategy.

\begin{figure}[htbp]
	\centering
	\includegraphics[width=0.8\textwidth]{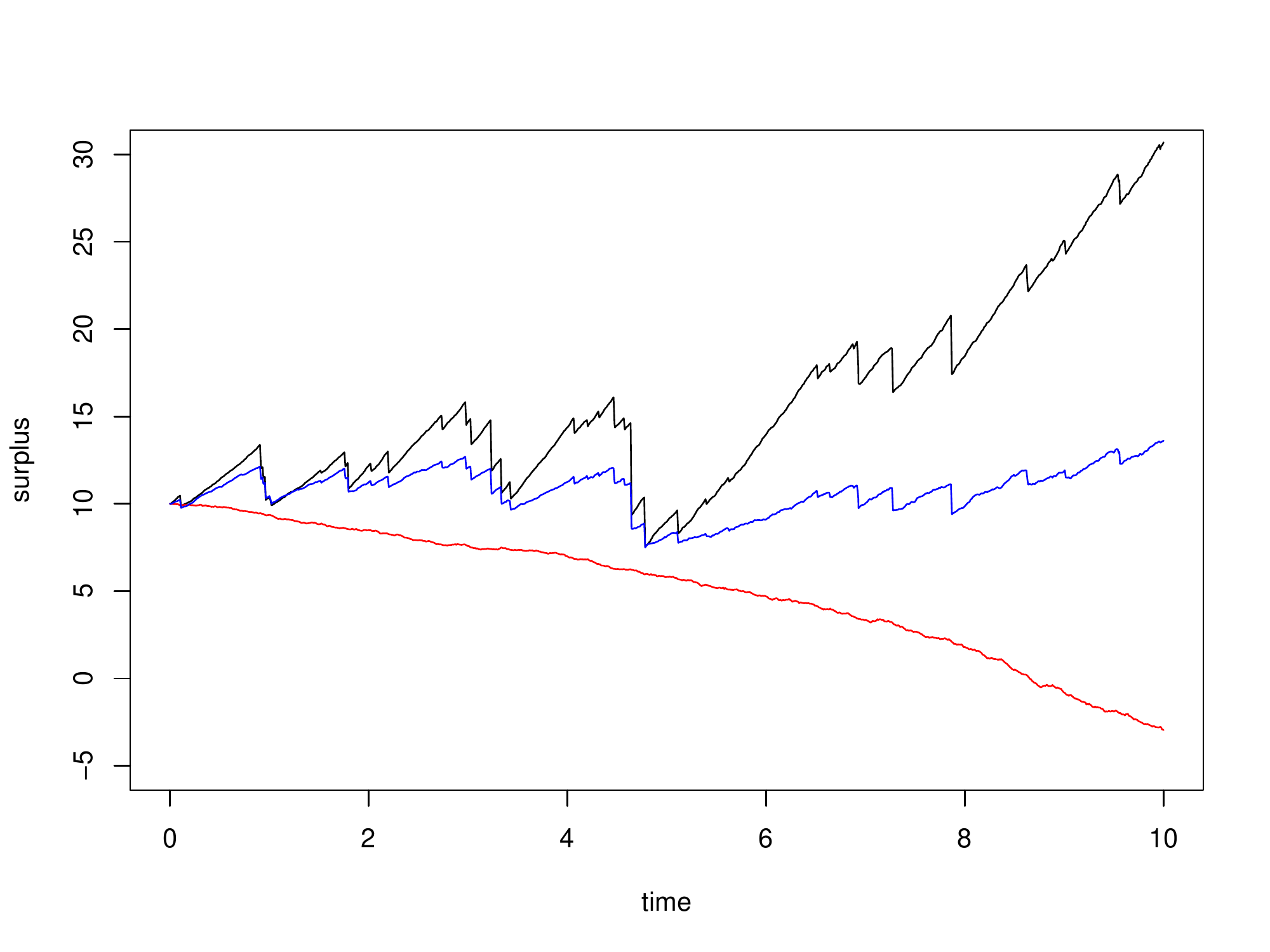}
	\caption[Paths of the surplus process for three different insurance strategies]{Paths of the surplus process  in case of full reinsurance (red), constant retention level of 0.5 (blue) and the reinsurance strategy $(b^\star_{u(p_{t-}),w(q_{t-})}(t))_{t\in[0,T]}$ with  $u(p)\defeq \sum_{k=1}^m\lambda_k p_k$ and $w(q)\defeq((\beta_D+q_D)/\normbetaq)_\DD$ (black).}
	\label{figDir:surplus}
\end{figure}

\section{Appendix}\label{sec:app}
\subsection{The Generalized Clark Gradient}
The following definition and results are taken from Section~2.1 in \cite{Clarke1983}.

\begin{definition}[\cite{Clarke1983}, p.\,25]\label{def:gendirder}
	Let $x\in\R^n$ be a given point and let $v\in\R^n$. Moreover, let $f$ be Lipschitz near $x$. Then the \emph{generalized directional derivative} of $f$ at $x$ in the direction $v$, denoted by $f^\circ(x;v)$, is defined by
	\begin{equation*}
	f^\circ(x;v) = \limsup_{y\to x, h\downarrow 0}\frac{f(y+h\,v)-f(y)}{h}.
	\end{equation*}
\end{definition}

\begin{definition}[\cite{Clarke1983}, p.\,27]\label{def:Clarksub}
	Let $f$ be Lipschitz near $x$. Then the \emph{generalized Clarke gradient} of $f$ at $x$, denoted by $\partial^C f(x)$, is given by
	\begin{equation*}
	\partial^Cf(x)\defeq\big\{\xi\in\R^n: f^\circ(x;v)\ge \xi^\top v\;\forall\;v\in\R^n\big\}.
	\end{equation*}
\end{definition}

%\begin{proposition}[\cite{Clarke1983}, Prop.\,2.1.2]\label{pr:propClarke}
%	Let $f$ be Lipschitz near $x$. Then $\partial^Cf(x)$ is non-empty, convex and a compact subset of $\R^n$.
%\end{proposition}

In the following, we denote by $D$ the differential operator taking the partial derivative of the function $f$.

\begin{proposition}[\cite{Clarke1983}, Prop.\,2.2.4]\label{pr:propClarkediff}
	If $f$ is strictly differentiable at $x$, then $f$ is Lipschitz near $x$ and $\partial^C f(x) = \{Df(x)\}$. Conversely, if $f$ is Lipschitz near $x$ and $\partial^C f(x)$ reduces to a singleton $\{\zeta\}$, then $f$ is strictly differentiable at $x$ and $Df(x)=\zeta$.
\end{proposition}

In what follows we denote by $\Omega_f$ the set of point at which the function $f$ is is not differentiable.

\begin{theorem}[\cite{Clarke1983}, Thm.\,2.5.1]\label{th:genCgco}
	Let $f$ be Lipschitz near $x$ and let $S$ be an arbitrary set of Lebesgue-measure $0$ in $\R^n$. Then
	\begin{equation*}
	\partial^C f(x) = co\Big\{\lim_{n\to\infty} \nabla f(x_n): x_n\to x, x_n\notin S, x_n\notin\Omega_f\Big\}.
	\end{equation*}
\end{theorem}

\subsection{Auxiliray Results}
Detailed calculations can be found in \cite{gl20}.
\begin{lemma}\label{leDir:characG}
	Suppose that $(\xi,b)\in U[0,T]$ is an arbitrary strategy and $h:\domaingDir\to(0,\infty)$ is a bounded function such that $t\mapsto h(t,p,q)$ and $t\mapsto h(t,\phi(t,p),q)$ are absolutely continuous on $[0,T]$ for all $(p,q)\in\Delta_m\times\N_0^\ell$ as well as $p\mapsto h(t,p,q)$ is concave for all $(t,q)\in[0,T]\times\N_0^\ell$. Then, the function $G:\domainVDir\to\R$ defined by
	\begin{equation*}
	G(t,x,p,q) \defeq -e^{-\alpha x e^{r(T-t)}}\,h(t,p,q)
	\end{equation*}
	satisfies
	\begin{equation*}
	\dif G(t,\Xxib_t,p_t,q_t)  = -e^{-\alpha \Xxib_t e^{r(T-t)}}\,\mathcal{H} h(t,p_t,q_t;\xi_t,b_t)dt + \dif\eta^{\xi,b}_t, \quad t\in[0,T],
	\end{equation*}
	where $\stprT{\eta^{\xi,b}}$ is a $\fG$-martingale and we set $\mathcal{H} h(t,p,q;\xi,b)$ zero at those points $(t,p,q)$ where $D h$ does not exist.
\end{lemma}

\begin{proof}
	 Let $\xib\in U[0,T]$ and $h:\domaingDir\to(0,\infty)$ be some function satisfying the conditions stated in the lemma, where $0<K_0<\infty$ is some constant which bounds $h$, i.e. $|h(t,p,q)|\le K_0$ for all $(t,p,q)\in\domaingDir$. Furthermore, we set
	\begin{equation}\label{eq:f}
	G(t,x,p,q) \defeq -e^{-\alpha x e^{r(T-t)}}\,h(t,p,q)\quad\text{and}\quad f(t,x)\defeq -e^{-\alpha x e^{r(T-t)}},
	\end{equation}
	for any $(t,x,p,q)\in\domainVDir$. Let us fix $t\in[0,T]$. Applying the product rule  to $G\big(t,\Xxib_t,p_t,q_t\big)=f\big(t,\Xxib_t\big)h(t,p_t,q_t)$, we get
	\begin{align*}
	\dif G\big(t,\Xxib_t,p_t,q_t\big) 
	&= h(t,p_{t-},q_{t-})df\big(t,\Xxib_t\big) + f\big(t,\Xxib_{t-}\big) d h(t,p_t,q_t)  + \dif\big[f\big(\cdot,\Xxib_\cdot\big),h(\cdot,p_\cdot,q_\cdot)\big]_t 
	\end{align*}
	and hence, 
	  \begin{equation}\label{eqprDir:G}
	\begin{aligned}
	&\dif G\big(t,\Xxib_t,p_t,q_t\big) \\
	&= f\big(t,\Xxib_t\big)h(t,p_t,q_t)\bigg(\alpha e^{r(T-t)}\Big(\frac12\alpha\sigma^2 e^{r(T-t)}\xi_t^2 - (\mu-r)\xi_t - c(b_t)\Big) \\
	&\quad + \widehat{\Lambda}_t\Dsum\fracbetaCt\intdalt\exp\bigg\{\alpha b_t e^{r(T-t)}\idsum y_i\ind{D}(i)\bigg\}F(dy)-\widehat{\Lambda}_t\bigg)\dif t \\
	&\quad-f\big(t,\Xxib_{t-}\big) h(t,p_{t-},q_{t-}) \alpha\,\sigma\, e^{r(T-t)} \xi_tdW_t \\
	&\quad+\intEd\!\! f\big(t,\Xxib_{t-}\big) h(t,p_{t-},q_{t-})\bigg(\!\exp\bigg\{\alpha b_t e^{r(T-t)}\idsum y_i\ind{z}(i)\bigg\}\!-\!1\!\bigg)\hat{\Psi}(\dtyz) \\
	&\quad+f\big(t,\Xxib_{t}\big) \bigg( D h(t,p_t,q_t) - \widehat{\Lambda}_t h(t,p_t,q_t)+ %\\
%	&\qquad\qquad
\widehat{\Lambda}_t\Dsum\fracbetaqt h\big(t,J(p_t),v(q_t,D)\big)\bigg)\dif t \\
	&\quad+\int_\PD f\big(t,\Xxib_{t-}\big)\Big(h\big(t,J(p_{t-}),v(q_{t-},z)\big)-h(t,p_{t-},q_{t-})\Big)\hat{\Psi}(dt,\R_+,dz) \\
	&\quad + \dif\big[f\big(\cdot,\Xxib_\cdot\big),h(\cdot,p_\cdot,q_\cdot)\big]_t.
	\end{aligned}
	\end{equation}
% Since $h(\cdot,p_\cdot,q_\cdot)$ is an FV process, it holds
%\begin{equation*}
%\big[f\big(\cdot,\Xxib_\cdot\big)+h(\cdot,p_\cdot,q_\cdot)\big]^c \equiv \big[f\big(\cdot,\Xxib_\cdot\big)\big]^c\quad\text{und}\quad \big[h(\cdot,p_\cdot,q_\cdot)\big]^c\equiv 0,
%\end{equation*}
%and, consequently, 
%\begin{align*}
%&\big[f\big(\cdot,\Xxib_\cdot\big),h(\cdot,p_\cdot,q_\cdot)\big]_t \\
%&= f\big(0,\Xxib_0\big)h(0,p_0,q_0)  + \stsum f\big(s,\Xxib_s\big)\big(h(s,p_s,q_s)-h(s,p_{s-},q_{s-})\big) \\
%&\quad - \stsum f\big(s,\Xxib_{s-}\big)\big(h(s,p_s,q_s)-h(s,p_{s-},q_{s-})\big) \\
%&= f\big(0,\Xxib_0\big)h(0,p_0,q_0) + \int_0^t\intEd f\Big(s,\Xxib_{s-}-b_s\idsum y_i\ind{z}(i)\Big)\times \\
%&\qquad \Big(h\big(s,J(p_{s-}),v(q_{s-},z)\big)-h(s,p_{s-},q_{s-})\Big)\Psi(\dsyz) \\
%&\quad- \int_0^t\int_\PD f\big(s,\Xxib_{s-}\big)\Big(h(s,J(p_{s-}),v(q_{s-},z))-h(s,p_{s-},q_{s-})\Big)\Psi(ds,\R_+,dz).
%\end{align*}
Using the introduced compensated random measure $\hat{\Psi}$  the variation becomes
\begin{align*}
&\big[f\big(\cdot,\Xxib_\cdot\big),h(\cdot,p_\cdot,q_\cdot)\big]_t \\
&= f\big(0,\Xxib_0\big)h(0,p_0,q_0) +\int_0^t \int_E f\big(s,\Xxib_{s-}\big)\exp\bigg\{\alpha\, b_s\, e^{r(T-s)}\idsum y_i \ind{z}(i)\bigg\}\times \\
&\qquad \Big(h\big(s,J(p_{s-}),v(q_{s-},z)\big)-h(s,p_{s-},q_{s-})\Big)\hat{\Psi}(\dsyz) \\
&\quad- \int_0^t\int_\PD f\big(s,\Xxib_{s-}\big)\Big(h\big((s,J(p_{s-}),v(q_{s-},z)\big)-h(s,p_{s-},q_{s-})\Big)\hat{\Psi}(ds,\R_+,dz) \\
&\quad + \int_0^t \widehat{\Lambda}_s\, f\big(s,\Xxib_s\big)\Dsum \fracbetaqs. %\intdalt\!\exp\bigg\{\alpha\, b_s\, e^{r(T-s)}\idsum y_i \ind{D}(i)\bigg\}\!F(dy)\times\\
%&\qquad \Big(h\big(s,J(p_{s}),v(q_{s},D)\big)-h(s,p_{s-},q_{s-})\Big)ds \\
%&\quad - \int_0^t \widehat{\Lambda}_s\,f\big(s,\Xxib_s\big)\Dsum\fracbetaCs h(s,J(p_s),v(q_{s},D))ds \\
%&\quad +\int_0^t \widehat{\Lambda}_s\, f\big(s,\Xxib_s\big) h(s,p_s,q_s)ds.
\end{align*}
  Substituting this into \eqref{eqprDir:G}, we obtain
\begin{align*}
&\dif G\big(t,\Xxib_t,p_t,q_t\big) \\
&=f\big(t,\Xxib_t\big)\bigg(- \alpha\, e^{r(T-t)}h(t,p_t,q_t)\Big((\mu-r)\,\xi_t+c(b_t)-\frac12\alpha\,\sigma^2\, e^{r(T-t)}\xi_t^2\Big) \\
& + \widehat{\Lambda}_t\Dsum\fracbetaCt h(t,J(p_t),v(q_t,D))\intdalt\!\exp\bigg\{\alpha\, b_t\,e^{r(T-t)} \idsum y_i\ind{D}(i)\bigg\}F(dy)\\
& - \widehat{\Lambda}_t\,h(t,p_t,q_t) + Dh(t,p_t,q_t)\bigg) \dif t \\
& - f\big(t,\Xxib_{t-}\big)\,h(t,p_{t-},q_{t-})\,\alpha\, \sigma\, e^{r(T-t)} \xi_t dW_t - f\big(t,\Xxib_{t-}\big)\,h(t,p_{t-},q_{t-})d\hat{N}_t \\
& + \int_E f\big(t,\Xxib_{t-}\big) \exp\bigg\{\alpha\, b_t\, e^{r(T-t)} \idsum y_i\ind{z}(i)\bigg\}h(t,J(p_{t-}),v(q_{t-},z))\hat{\Psi}(dt,d(y,z)),
\end{align*}
where $  \hat{N}_t \defeq N_t - \int_0^t\hat{\Lambda}_{s}ds.$
  Therefore, by definition of the operator $\mathcal{H}$ given in~\eqref{eqDir:ccH}, we have
\begin{equation*}
\dif G\big(t,\Xxib_t,p_t,q_t\big) = f\big(t,\Xxib_{t}\big)\mathcal{H} h(t,p_t,q_t;\xi_t,b_t)dt + \dif\eta^{\xi,b}_t,
\end{equation*}
where $\eta^{\xi,b}_t \defeq \etaxibhat_t - \etaxibbar_t -  \etaxibtilde_t$
with
\begin{align*}
\etaxibhat_t &\defeq \int_0^t\int_E f\big(s,\Xxib_{s-}\big) \exp\bigg\{\alpha\, b_s\, e^{r(T-s)} \idsum y_i\ind{z}(i)\bigg\}\times 
 h\big(s,J(p_{s-}),v(q_{s-},z)\big)\hat{\Psi}(\dsyz), \\
\etaxibbar_t &\defeq \int_0^t f\big(s,\Xxib_{s-}\big) \,h(s,p_{s-},q_{s-})d\hat{N}_s,\\
\etaxibtilde_t &\defeq \int_0^t f\big(s,\Xxib_{s-}\big)h(s,p_{s-},q_{s-})\,\alpha\, \sigma\, e^{r(T-s)} \xi_s dW_s.
\end{align*}
 To complete the proof we need to show that the introduced processes are $\fG$-martingales on $[0,T]$. For details we refer the reader to \cite{gl20}.
\end{proof}

\begin{lemma}\label{le:fbounded}
	Let $f:[0,T]\times\R\to\R$ be the function defined by~\eqref{eq:f}.
	Furthermore, let $\xib\in U[0,T]$ where $\xi$ is only adapted w.r.t.\ $W$ and $b$ is predictable w.r.t.\ the filtration generated by $\Psi$ and let $\prT{L^{\xi,b}}$ be the density process given by
	\begin{align*}
	L^{\xi,b}_t &\defeq \exp\bigg\{ -\int_0^t \alpha\,\sigma\,e^{r(T-s)}\xi_s dW_s -\frac12 \int_0^t \alpha^2\,\sigma^2\,e^{2r(T-s)}\xi_s^2 ds \\
	&\;\quad+\int_0^t\int_E \alpha\,b_{s}\,e^{r(T-s)}\idsum y_i\ind{z}(i)\,\Psi(\dsyz) +\int_0^t\hat{\Lambda}_s ds \\
	&\;\quad -\int_0^t \hat{\Lambda}_s\Dsum \fracbetaCs \intdalt\exp\bigg\{\alpha\,b_{s}\,e^{r(T-s)}\idsum y_i\ind{D}(i)\bigg\} \,F(\dif y) ds\bigg\}.
	\end{align*}
	
	Then there exists a constant $0<K_2<\infty$ such that 
	\begin{equation*}
	\frac{\big|f(t,\Xxib_t)\big|}{L^{\xi,b}_t}\le K_2\quad\Pas
	\end{equation*}
	for all $t\in[0,T]$.
\end{lemma}

\begin{proof}
	Fix $t\in[0,T]$ and $\xib\in U[0,t]$. We obtain
	\begin{align*}
	&\frac{\big|f(t,\Xxib_t)\big|}{L^{\xi,b}_t} = \exp\bigg\{-\alpha x_0 e^{rT} -\int_0^t \alpha e^{r(T-s)}\Big( (\mu-r)\xi_s +c(b_s) - \frac12\alpha\sigma^2 e^{r(T-s)}\xi_s^2\Big) ds  \\
	&\;+ \int_0^t \hat{\Lambda}_s\Dsum \fracbetaCs \intdalt\exp\bigg\{\alpha\,b_{s}\,e^{r(T-s)}\idsum y_i\ind{D}(i)\bigg\} \,F(\dif y) ds -\int_0^t \hat{\Lambda}_s ds\bigg\} \\
	&\le\exp\bigg\{\!\bigg(\alpha e^{|r|T}\big(|\mu-r|K+(2+\eta+\theta)\kappa\big)+\frac12\alpha^2\,\sigma^2\, e^{2|r|T}K^2 + \lambda_m M_F\big(\alpha e^{|r|T}\big)\!\bigg)T\bigg\}=: K_2,
	\end{align*}
	where $0<K_2<\infty$ is independent of $t\in[0,T]$ as well as $\xib$.
\end{proof}

The following result can be found in \cite{Mitrinovic1993}.

\begin{lemma}\label{le:ineqsum}
	Let $\alpha_1\le \ldots \le \alpha_n$ and $\beta_1\le\ldots\le\beta_n$ be real numbers and $(p_1,\ldots,p_n)\in\Delta_n$. Then
	\begin{equation*}
	\sum_{j=1}^n p_j\alpha_j\beta_j \ge \sum_{j=1}^n p_j\alpha_j\sum_{k=1}^n p_k\beta_k.
	\end{equation*}
\end{lemma}

\subsection{Proofs}

For convenience we introduce the operator $D$ acting on functions $h:\domaingDir\to(0,\infty)$ by
\begin{equation}\label{notDir:Dh}
D h(t,p,q) \defeq h_t(t,p,q) + \sum_{j=1}^m h_{p_j}(t,p,q)\,p_j\,\Big(\sum_{k=1}^m\lambda_k\,p_k-\lambda_j\Big)
\end{equation}
for all functions $h$, where the right-hand side exists.
Furthermore, we define an operator $\mathcal{H}$ 
%acting on functions $v:\domaingDir\to(0,\infty)$ and $(\xi,b)\in\R\times[0,1]$ by
\begin{equation}\label{eqDir:ccH}
\mathcal{H} h(t,p,q;\xi,b) \defeq \mathcal{L} h(t,p,q;\xi,b) + D\,h(t,p,q)
\end{equation}
for all functions $h:\domaingDir\to(0,\infty)$ and $(\xi,b)\in \R\times [0,1]$, where the right-hand side is well-defined.
Using this notation, the generalized HJB equation~\eqref{eqDir:HJBg} can be written as
\begin{equation}\label{eq:HHJB}
0= \inf_{(\xi,b)\in\R\times[0,1]}\{\mathcal{H} g(t,p,q;\xi,b)\}
\end{equation}
at those points $(t,p,q)$ with existing $Dg(t,p,q)$.
% such that $\partial^C g_q(t,p) = \{\tilde{D} g_q(t,p)\}$, where $\tilde{D}$ denotes the usual differential operator taking the partial derivatives.

\begin{proof}[Proof of Theorem~\ref{thDir:veri}]
   Let $h:\domaingDir\to(0,\infty)$ be a function satisfying the conditions stated in the theorem. Note that every Lipschitz function is also absolutely continuous. We set, for any $(t,x,p,q)\in\domaingDir$,
   \begin{equation*}
   f(t,x)\defeq -e^{-\alpha x e^{r(T-t)}}\quad\text{and}\quad G(t,x,p,q) \defeq f(t,x)\,h(t,p,q).
   \end{equation*}
   Let us fix $t\in[0,T]$ and $(\xi,b)\in U[t,T]$.
   From Lemma~\ref{leDir:characG} in the appendix, it follows 
   \begin{equation}\label{eqp:GTwfs}
   G(T,\Xxib_T,p_T,q_T) = G(t,\Xxib_t,p_t,q_t) + \int_t^T\!\! f(s,\Xxib_s)\,\mathcal{H} h(s,p_s,q_s;\xi_s,b_s)ds + \eta^{\xi,b}_T - \eta^{\xi,b}_t ,
   \end{equation}
   where $\stprT{\eta^{\xi,b}}$ is a $\fG$-martingale and we set $\mathcal{H} h(s,p_s,q_s;\xi,b)$ to zero at those points $\stT$ where $Dh$ does not exist. Note that $h$ is partially differentiable w.r.t.\ $t$ almost everywhere in the sense of the Lebesgue measure according to the absolute continuity of $t\mapsto h(t,p,q)$ for all $(p,q)\in\Delta_m\times \N_0^\ell$.
   The generalized HJB equation~\eqref{eq:HHJB} implies 
   \begin{equation*}
   \mathcal{H} h(s,p_s,q_s;\xi_s,b_s)\ge0\quad\stT.  %\text{for almost all } \stT.
   \end{equation*}
   As a consequence
   \begin{equation*}
   \int_t^T f(s,\Xxib_s)\,\mathcal{H} h(s,p_s,q_s;\xi_s,b_s) ds\le 0,
   \end{equation*}
   due to the negativity of  $f$. Thus, by~\eqref{eqp:GTwfs}, we get
   \begin{equation}\label{eq:GGeta}
   G(T,\Xxib_T,p_T,q_T)\le G(t,\Xxib_t,p_t,q_t) + \eta^{\xi,b}_T-\eta^{\xi,b}_t.
   \end{equation}
   Using the boundary condition~\eqref{eqDir:hHJBbcond}, we obtain
   \begin{equation*}
   G(T,x,p,q) = f(T,x)\,h(T,p,q) = f(T,x) = -e^{-\alpha x}=U(x).
   \end{equation*}
   %	Hence
   %	\begin{equation*}
   %	U(\Xxib_T) = G(T,\Xxib_T,p_T,q_T) \le G(t,\Xxib_t,p_t,q_t) + \eta^{\xi,b}_T-\eta^{\xi,b}_t.
   %	\end{equation*}
   Now, we take the regular conditional expectation in \eqref{eq:GGeta} given $\Xxib_t=x$, $p_t=p$ and $q_t=q$ on both sides of the inequality which yields
   \begin{equation*}
   \Eop^{t,x,p,q}\big[U(\Xxib_T)\big]\le G(t,x,p,q)
   \end{equation*}
   since $\stprtT{\eta^{\xi,b}}$ is a $\fG$-martingale.
   Taking the supremum over all investment and reinsurance strategies $(\xi,b)\in U[t,T]$, we obtain 
   \begin{equation}\label{eqp:VG}
   V(t,x,p,q) \le G(t,x,p,q).
   \end{equation} 
   %	As already seen in~\eqref{eq:gHJBa}, we can rewrite the generalized HJB equation as
   %	\begin{equation}\label{eqpr:HJB}
   %	\begin{aligned}
   %	0 &= \inf_{(\xi,b)\in\R\times[0,1]}\{\mathcal{L} h(s,p,q;\xi,b)\} + \inf_{\varphi \in \partial^C h_q(s,p) }\Big\{ \varphi_0 + \sum_{j=1}^m \varphi_j p_j \Big( \sum_{k=1}^m \lambda_kp_k -\lambda_j\Big) \Big\} %\\
   %	&= - \lambda\,h(s,p) + \alpha\,e^{r(T-s)}h(s,p)\inf_{\xi\in\R} f_1(s,\xi) + \inf_{b\in[0,1]} f_2(s,p,b) +\inf_{\varphi \in \partial^C h_p(s) }\{ \varphi \},
   %	\end{aligned}
   %	\end{equation}
   %	where $f_1$ is defined by~\eqref{eq:fone} and $f_2$ by~\eqref{eq:ftwo}.
   To show equality, note that  $\xi^\star(s)$ given by~\eqref{eqDir:xistarfunc} %is the unique minimizer of $f_1$ on $\R$ 
   and $b^\star(s,p,q)$ given by~\eqref{eqDir:bstarfunc} (with $g$ replaced by $h$ in $A(s,p,q)$ and $B(s,p,q)$) are the unique minimizer of the HJB equation \eqref{eqDir:HJBg}. 
   %Since $\partial^C h_q(s,p)$ is a compact subset of $\R^{m+1}$ (c.f.\ Prop.~\ref{pr:propClarke}), we know that the second infimum in \eqref{eqDir:HJBg} is attained which we denote by $\varphi^\star(s,p_s,q_s)$.
   Therefore, 
   \begin{equation*}
   \mathcal{L} h(s,p_s,q_s;\xi^{\star}(s),b^\star(s,p_s,q_s)) + \inf_{\varphi\in\partial^C g_q(t,p)}\bigg\{\varphi_0 + \sum_{j=1}^m\varphi_j\,p_j\bigg(\sum_{k=1}^m \lambda_k\,p_k-\lambda_j\bigg)\bigg\} = 0.
   \end{equation*}
   %	We set $b^{\star}_{\lambda,F}(s)\defeq\blamF(s-,p_{s-},q_{s-})$ for every $\stT$. %It is now routine to check that $(\xi^{\star},b^{\star})$  is an admissible investment and reinsurance strategy.
   %	First, we observe that $\stprtTalt{\xistar}$ is a continuous, bounded by $\frac{|\mu-r|}{\sigma^2}\frac{1}{\alpha}$ and deterministic (\ie $\fF^W$-adapted) process. Secondly, we note that $(t,\omega)\mapsto h(t,p_{t-}(\omega))$ and $(t,\omega)\mapsto h(t,J(p_{t-}(\omega),D))$ are $\cP(\fF^\Psi)$-measurable for all $\DD$, compare the arguments in proof of Lemma~\ref{le:characG}. This implies that $(t,\omega)\mapsto \AlamF(t,p_{t-}(\omega))$ and $(t,\omega)\mapsto \BlamF(t,p_{t-}(\omega))$ defined by~\eqref{eq:ABterm} with $g$ replaced by $h$ as well as the $\cP(\fF^\Psi)$-measurability of right-hand side of~\eqref{eq:hb}. 
   %	Consequently, the unique root $(t,\omega)\mapsto \blamF(t,p_{t-}(\omega))$ is $\cP(\fF^\Psi)$-measurable. In summary, $(t,\omega)\mapsto \bstar(t,p_{t-}(\omega))$ is $\cP(\fF^\Psi)$-measurable. That is, $\stprtTalt{\bstar}$ is an $\fF^\Psi$-predictable (in particular, $\fG$-predictable and $\fF^\Psi$-adapted), $[0,1]$-valued process and thus $(\xistar,\bstar)\in\cUtilde[t,T]$. 
   So we can deduce that %$\Vtilde\equiv V$ and \marginnote{since $\ccH h(s,\phat_{s},\xistar_s,\bstar_s)$ is zero at those points $\stT$ where the $h_t(s,\phat_s)$ does not exist}
   \begin{equation*}
   \mathcal{H} h(s,p_{s},q_s;\xi^{\star}(s),b^{\star}(s)) = 0,\quad \stT.
   \end{equation*}
   This implies 
   \begin{equation*}
   \int_t^T f(s,\Xxibstar_s)\,\mathcal{H} h(s,p_s,q_s;\xi^{\star}(s),b^{\star}(s))ds = 0.
   \end{equation*}
   Consequently,
   \begin{equation*}
   U(\Xxibstar_T)= G(T,\Xxibstar_T,p_T,q_T) = G(t,X^{\xi^\star,b^\star}_t,p_t,q_t) + \eta^{\xi^\star,b^\star}_T - \eta^{\xi^\star,b^\star}_t.
   \end{equation*}
   Again, taking the regular conditional expectation given $\Xxibstar_t=x$ and $p_t=p$  on both sides then yields
   \begin{equation*}
   \Eop^{t,x,p}\big[U(\Xxibstar_T)\big] = G(t,x,p,q)= -e^{-\alpha x e^{r(T-t)}}h(t,p,q)
   \end{equation*}
   %	That is,
   %	\begin{equation*}
   %	\Vxibstar(t,x,p) = G(t,x,p) = -e^{-\alpha x e^{r(T-t)}}h(t,p),
   %	\end{equation*}
   %	which implies
   %	\begin{equation*}
   %	V(t,x,p) = \sup_{(\xi,b)\in\UtildetT}V^{\xi,b}(t,x,p) = V^{\xi^{\star},b^{\star}}(t,x,p) =  -e^{-\alpha x e^{r(T-t)}}h(t,p).
   %	\end{equation*} 
   %	From~\eqref{eqpr:HJB} follows that the optimal investment and reinsurance strategy are independent since $f_1$ depends not on the filter $\stpr{p}$. Moreover, we have seen in Section~\ref{sec:optinv} that
   %	\begin{equation*}
   %	|\xistar(t)|=\bigg|\arginf_{\xi\in\R} f_1(t,\xi) \bigg|= \bigg|\frac{\mu-r}{\sigma^2}\frac{1}{\alpha}e^{-r(T-t)}\bigg| \le \frac{|\mu-r|}{\sigma^2}\frac{1}{\alpha},
   %	\end{equation*}
   %	\ie the optimal investment strategy is continuous, bounded and deterministic (in particular $\fF^W$-adapted). The definition of $f_2$ implies that the optimal reinsurance strategy is $\fF^\Psi$-predictable, compare the first order condition~\eqref{eqDir:ftwofoc}.
   %	Therefore,
   %	\begin{equation*}
   %	\sup_{\xib\in\UtildetT}V^{\xi,b}(t,x,p) = \sup_{\xib\in\UtT}V^{\xi,b}(t,x,p)
   %	\end{equation*}
   %	and thus
   %	\begin{equation*}
   %	V(t,x,p) =  -e^{-\alpha x e^{r(T-t)}}h(t,p)
   %	\end{equation*}
   and the proof is complete.
\end{proof} 

\begin{proof}[Proof of Lemma \ref{leDir:propg}]
	\begin{enumerate}
		\item By the definition of $g$ we immediately obtain $g\ge 0$. In order to show that $g$ is bounded from above we consider the strategy $(\xi,b)\equiv (0,0)$. What remains is
		$$g^{0,0}(t,p,q)=  \Eop^{t,p,q}\bigg[\exp\bigg\{-\int_t^T \alpha\, e^{r(T-s)}(\eta-\theta)\kappa ds\bigg\}\bigg]<\infty.$$ To show that $g>0$ we use the change of measure introduced in Lemma~\ref{le:fbounded}. With its help it is possible to prove that $g^{\xi,b}(t,p,q)$ is bounded from below by a positive constant independent of $\xi$ and $b$.
		%	\item This statement follows immediately from the definition of $g^{\xi,b}$ given in~\eqref{eqDir:gxib}.
		\item Follows by conditioning. 
		\item  Let us fix $t\in[t,T]$ and $q=(q_1,\ldots,q_m)\in\Delta_m$ and $\beta\in(0,1)$. Suppose $p,p' \in\Delta_m$. We obtain
		\begin{align*}
		&g(t,\beta p + (1-\beta)p',q) \\
		&\quad= \inf_{\xib\in\UtT}  \sum_{j=1}^m (\beta p_j + (1-\beta)p'_j) g^{\xi,b}(t,x,e_j) \\
		&\quad= \inf_{\xib\in\UtT}\left(  \beta\sum_{j=1}^m p_j g^{\xi,b}(t,x,e_j) + (1-\beta) \sum_{j=1}^m p'_j g^{\xi,b}(t,x,e_j)\right) \\
		&\quad\ge \beta \inf_{\xib\in\UtT}  \sum_{j=1}^m p_j g^{\xi,b}(t,x,e_j) + (1-\beta) \inf_{\xib\in\UtT}  \sum_{j=1}^m q_j g^{\xi,b}(t,x,e_j)  \\
		&\quad=\beta g(t,p,q) + (1-\beta)g(t,p',q),
		\end{align*}
		for all $t\in[0,T]$ and $x\in\R$.
		\item The Lipschitz condition is proven in much the same way as in \cite[Lemma 6.1 d)]{BaeuerleRieder2007}.
		\item The Lipschitz condition is proven in much the same way as in \cite[Lemma 6.1 e)]{BaeuerleRieder2007}. \qedhere
%		\item	Fix $(t,p,q)\in M$. Using the Lipschitz properties  of $g$   we obtain
%		\begin{align*}
%		&\big|D\,g(t,p,q)\big| \\
%		&\le \bigg|\lim_{h\downarrow 0} \frac{1}{h}\bigg(g(t+h,p,q)-g(t,p,q)\bigg)\bigg| \\
%		&\quad+\sum_{j=1}^m \bigg|\lim_{h\downarrow 0}\frac{1}{h}\bigg(g(t,p_1,\ldots,p_{j-1},p_j+h,p_{j+1},\ldots,p_m,q)-g(t,p,q)\bigg)p_j\bigg|\times\\
%		&\qquad\bigg|\sum_{k=1}^m p_k\,\lambda_k-\lambda_j\bigg| \\
%		&\le \lim_{h\downarrow 0}\frac{1}{h}L_1\,h + m\,\Big|\lim_{h\downarrow 0}\frac{1}{h}L_2\,h\Big|\,(m+1)\lambda_m 
%		\le L_1 + m\,L_2\,(m+1)\lambda_m 
%		\end{align*}
%		for some constants $0<L_1<\infty$ and $0<L_2<\infty$.
%		\item Fix $(t,p,q)\in\domaingDir$ and $\xib\in[-K,K]\times[0,1]$. It follows from the definition of $\mathcal{L}$ given in~\eqref{eqDir:L} statement~(iii) that
%		\begin{align*}
%		\big|\mathcal{L} g(t,p,q;\xi,b)\big| &\le K_3\bigg(m\lambda_m + \alpha\,e^{rT}\Big(|\mu-r|K+(2+\eta+\theta)\,\kappa + \frac12\alpha\,\sigma^2e^{rT}K^2\Big) \\
%		&\quad+ m\lambda_m\,M_F\big(\alpha\,e^{rT}\big)\bigg) =: K_5 
%		\end{align*}
%		\item  Fix $(t,p,q)\in\domaingDir$. In the same way as in the proof the previous statement, the following results arise by taking account of~\eqref{eq:fxistar},~\eqref{eqDir:gHJBa} and~\eqref{eqDir:ftwoa}:
%		\begin{align*}
%		&\big| \inf_{(\xi,b)\in[-K,K]\times[0,1]}\mathcal{L} g(t,p,q;\xi,b)\big| \\
%		&\le K_3\bigg(m\lambda_m+\alpha\,e^{rT}\frac12\frac{(\mu-r)^2}{\sigma^2}\frac{1}{\alpha}+\alpha\,e^{rT}(2+\eta+\theta)\,\kappa+m\lambda_m\,M_F(\alpha\,e^{rT})\bigg)=: K_6.
%		\end{align*}
	\end{enumerate}
\end{proof}

\begin{proof}[Proof of Theorem~\ref{thDir:existanceHJB}]
   Fix $t\in[0,T)$ and $\xib\in U[t,T]$ and set $f(t,x) := -e^{-\alpha x e^{r(T-t)}}$, $x\in \R.$
   %	\begin{equation}\label{eq:f}
   %	f(t,x) \defeq -e^{-\alpha x e^{r(T-t)}},\quad x\in \R.
   %	\end{equation}
   Let $\tau$ be the first jump time of $\Xxib$ after $t$
   %, $\tau'$ the last jump of $\Xxib$ before $\tau$ 
   and $t'\in(t,T]$. It follows from Lemma~\ref{leDir:characG} and Lemma \ref{leDir:propg} that 
   \begin{equation}\label{eqprDir:Vtilde}
   \begin{aligned}
   &V(\tau\wedge t',\Xxib_{\tau\wedge t'},p_{\tau\wedge t'},q_{\tau\wedge t'}) \\
   &= V(t,\Xxib_t,p_t,q_t) + \int_t^{\tau\wedge t'} f(s,\Xxib_s)\,\mathcal{H}g(s,p_s,q_s;\xi_s,b_s)ds + \eta^{\xi,b}_{\tau\wedge t'} - \eta^{\xi,b}_t,
   \end{aligned}
   \end{equation}
   where $\stprT{\eta^{\xi,b}}$ is a $\fG$-martingale and we set $\mathcal{H}g(s,p_s,q_s;\xi_s,b_s)$ to zero at those $s\in[t,T]$ where the $D g(s,p_s,q_s)$ does not exist. 
   For any $\varepsilon>0$ we can construct  a strategy $(\xi^\varepsilon,b^\varepsilon)\in U[t,T]$ with $(\xi^\varepsilon_s,b^\varepsilon_s)=(\xi_s,b_s)$ for all $s\in[t,\tau\wedge t']$  from the continuity of $V$ such that
   \begin{align*}
   \Eop^{t,x,p,q}\Big[V(\tau\wedge t',\Xxib_{\tau\wedge t'},p_{\tau\wedge t'},q_{\tau\wedge t'})\Big]
   &\le \Eop^{t,x,p,q}\Big[\Eop^{\tau\wedge t',\Xxib_{\tau\wedge t'},p_{\tau\wedge t'},q_{\tau\wedge t'}}\Big[U(X_T^{\xi^\varepsilon,b^\varepsilon})\Big]\Big] +\varepsilon \\
   &\le \Eop^{t,x,p,q}\Big[U(X_T^{\xi^\varepsilon,b^\varepsilon})\Big]+\varepsilon \le V(t,x,p,q)+\varepsilon.
   \end{align*}
   From the arbitrariness of $\varepsilon>0$ we conclude
   $$V(t,x,p,q) \ge \Eop^{t,x,p,q}\Big[V(\tau\wedge t',\Xxib_{\tau\wedge t'},p_{\tau\wedge t'},q_{\tau\wedge t'})\Big]. $$
   Using this statement and \eqref{eqprDir:Vtilde} we obtain 
   \begin{align*}
   0 &\ge \lim_{t'\downarrow t}\Eop^{t,x,p,q}\bigg[\frac{1}{t'-t}\int_t^{t'} f(s,\Xxib_s)\,\mathcal{H}g(s,p_s,q_s;\xi_s,b_s) ds \big| t'<\tau\bigg]\Pop^{t,x,p,q}(t'<\tau) \\
   &\quad + \lim_{t'\downarrow t}\Eop^{t,x,p,q}\bigg[\frac{1}{t'-t}\int_t^{\tau} f(s,\Xxib_s)\,\mathcal{H}g(s,p_s,q_s;\xi_s,b_s)ds\big| t'\ge\tau\bigg]\Pop^{t,x,p,q}(t'\ge\tau).
   \end{align*}
   We have 
   \begin{equation*}
   \Pop^{t,x,p,q}(\tau\le t') = \int \P_\lambda(\tau\le t')\,\Pi_\Lambda(\dif\lambda) = \sum_{j=1}^m\Big(1-e^{-\lambda(t'-t)}\Big)\pi_\Lambda(j).
   \end{equation*}
   Thus 
   \begin{equation*}
   \lim_{t'\downarrow t}\Pop^{t,x,p,q}(\tau\le t')=\sum_{j=1}^m\Big(1-\lim_{t'\downarrow t}e^{-\lambda(t'-t)}\Big)\pi_\Lambda(j)=0.
   \end{equation*}
   Consequently,
   \begin{equation*}
   0\ge \lim_{t'\downarrow t}\Eop^{t,x,p,q}\bigg[\frac{1}{t'-t}\int_t^{t'} f(s,\Xxib_s)\,\mathcal{H}g(s,p_s,q_s;\xi_s,b_s)ds\Ind{t'<\tau}\bigg].
   \end{equation*}
   %	Note that, due to Lemma~\ref{leDir:propg}~(vii) and~(viii) 
   %	\begin{align*}
   %	&\Eop^{t,x,p,q}\bigg[\frac{1}{t'-t}\int_t^{t'} f(s,\Xxib_s)\,\mathcal{H}g(s,p_s,q_s;\xi_s,b_s)ds\Ind{t'<\tau}\bigg] \\
   %	&\le \Eop^{t,x,p,q}\bigg[\frac{1}{t'-t}\int_t^{t'} |f(s,\Xxib_s)|(K_5+K_4)ds\bigg] <\infty\\
   %	&\le \frac{K_4+K_5}{t'-t}\int_t^{t'}\Eop^{t,x,p,q}_{\mathbb{Q}^{\xi,b}}\bigg[\frac{|f(s,\Xxib_s)|}{L^{\xi,b}_s}\bigg]ds \\
   %	&\le \frac{K_4+K_5}{t'-t}K_1(t'-t) = (K_4+K_5)\,K_1 <\infty
   %	\end{align*}
   %	where $\frac{d \mathbb{Q}^{\xi,b}}{d \Pop^{t,x,p,q} }=L^{\xi,b}$ and the last but one inequality is given in Lemma \eqref{le:fbounded}.
   By the dominated convergence theorem, we can interchange the limit and the expectation and %. That is,
   %	\begin{equation*}
   %	0\ge \Eop^{t,x,p,q}\bigg[\lim_{t'\downarrow t}\frac{1}{t'-t}\int_t^{t'} f(s,\Xxib_s)\,\mathcal{H}g(s,p_s,q_s;\xi_s,b_s)ds\Ind{t'<\tau}\bigg].
   %	\end{equation*}
   we obtain by the fundamental theorem of Lebesgue calculus and $\Ind{t'<\tau}\to1$ $\Pas$ for $t'\downarrow t$, 
   \begin{equation*}
   0\ge \Eop^{t,x,p,q}\bigg[f(t,\Xxib_t)\,\mathcal{H}g(t,p_t,q_t;\xi_t,b_t)\bigg].
   \end{equation*}
   From now on, let $\xib\in[-K,K]\times[0,1]$ and $\varepsilon>0$ as well as $(\bar{\xi},\bar{b})\in U[t,T]$ be a fixed strategy with $(\bar{\xi}_s,\bar{b}_s)\equiv(\xi,b)$ for $s\in[t,t+\varepsilon)$. Then
   \begin{equation*}
   0\ge \Eop^{t,x,p,q}\bigg[f(t,X^{\bar{\xi},\bar{b}}_t)\,\mathcal{H}g(t,p_t,q_t;\bar{\xi}_t,\bar{b}_t)\bigg] = f(t,x)\mathcal{H}g(t,p,q;\xi,b)
   \end{equation*}
   at those points $(t,p,q)$ where $D g(t,p,q)$ exists. Due to the negativity of $f$, we get
   \begin{equation*}
   0\le \mathcal{H}g(t,p,q;\xi,b).
   \end{equation*}
   We show next the inequality above if $D g$ does not exist. For this purpose, we denote by $M_q\subset[0,T]\times\Delta_m$ the set of points at which $\nabla g_q(t,p)$ exists for any $q\in\N_0^\ell$. On the basis of Theorem~\ref{th:genCgco}, we have, for any $q\in\N_0^\ell$, 
   \begin{equation*}
   \partial^C g_q(t,p) = co\Big\{\lim_{n\to\infty} \nabla g_q(t_n,p_n): (t_n,p_n)\to(t,p), (t_n,p_n)\in M_q\Big\}.
   \end{equation*}
   That is, for every $\varphi\in\partial^C g_q(t,p)\subset[0,T]\times\Delta_m$, there exists $u\in\N$ and $(\beta_1,\ldots,\beta_u)\in\Delta_u$ such that $\varphi = \sum_{i=1}^u \beta_i\,\varphi^i$, where $\varphi^i = \lim_{n\to\infty} \nabla g_q(t_n^i,p_n^i)$ for sequences $(t_n^i,p_n^i)_{n\in\N}$ with $\lim_{n\to\infty}(t_n^i,p_n^i)=(t,p)$ along   existing $\nabla g_q$. From what has already been proved, it can be concluded that, for any $i=1,\ldots,u$
   \begin{equation*}
   0 \le \mathcal{L}g(t_n^i,p_n^i,q;\xi,b)+g_t(t_n^i,q_n^i,q) + \sum_{j=1}^m g_{p_j}(t_n^i,p_n^i,q)(p_n^i)_j\bigg(\sum_{k=1}^m\lambda_k(p_n^i)_k -\lambda_j\bigg),
   \end{equation*}
   where $(p_n^i)_j$ denotes the $j$th component of the $m$-dimensional vector $p_n^i$. Thus, by the continuity of $t\mapsto g(t,p,q)$, $p\mapsto g(t,p,q)$ and $p\mapsto J(p)$, we get for $i=1,\ldots, u$
   \begin{align*}
   0 &\le \beta_i\mathcal{L}g(t,p,q;\xi,b)+\beta_i\lim_{n\to\infty}g_t(t_n^i,q_n^i,q) %\\	&\quad
   + \sum_{j=1}^m \beta_i\lim_{n\to\infty}g_{p_j}(t_n^i,p_n^i,q)p_j\bigg(\sum_{k=1}^m\lambda_k\,p_k -\lambda_j\bigg), 
   \end{align*}
   which yields
   \begin{align*}
   0 &\le \mathcal{L}g(t,p,q;\xi,b)+\sum_{i=1}^u\beta_i\lim_{n\to\infty}g_t(t_n^i,q_n^i,q) %\\	&\quad
   + \sum_{j=1}^m \sum_{i=1}^u\beta_i\lim_{n\to\infty}g_{p_j}(t_n^i,p_n^i,q)p_j\bigg(\sum_{k=1}^m\lambda_k\,p_k -\lambda_j\bigg) \\
   &= \mathcal{L}g(t,p,q,\xi,b)+ \varphi_0 + \sum_{j=1}^m \varphi_j\,p_j\bigg(\sum_{k=1}^m\lambda_k\,p_k -\lambda_j\bigg).
   \end{align*}
   
   Due to the arbitrariness of $\varphi\in\partial^C g_q(t,p)$ and $\xib\in[-K,K]\times[0,1]$, we obtain 
   \begin{equation*}
   0 \le \inf_{(\xi,b)\in[-K,K]\times[0,1]}\mathcal{L}g(t,p,q,\xi,b)+ \inf_{\varphi\in\partial^C g_q(t,p)}\bigg\{\varphi_0 + \sum_{j=1}^m\varphi_j\,p_j\bigg(\sum_{k=1}^m \lambda_k\,p_k-\lambda_j\bigg)\bigg\}.
   \end{equation*}
   Our next objective is to establish the reverse inequality. For any $\varepsilon>0$ and $0\le t<t'\le T$, there exists a strategy $({\xi^{\varepsilon,t^\prime},b^{\varepsilon,t^\prime}})\in U[t,T]$ such that
   \begin{equation*}
   V(t,x,p,q)-\varepsilon(t'-t)\le \Eop^{t,x,p,q}\Big[U\big(X_T^{\xi^{\varepsilon,t^\prime},b^{\varepsilon,t^\prime}}\big)\Big] 
   \le \Eop^{t,x,p,q}\Big[V\big(\tau\wedge t',X^{\xi^{\varepsilon,t^\prime},b^{\varepsilon,t^\prime}}_{\tau\wedge t'},p_{\tau\wedge t'},q_{\tau\wedge t'}\big)\Big].
   \end{equation*}
   Using Lemma \ref{leDir:characG} it follows
   \begin{equation*}
   -\varepsilon(t'-t)\le \Eop^{t,x,p,q}\bigg[\int_t^{\tau\wedge t'} f\big(s,X^{\xi^{\varepsilon,t^\prime},b^{\varepsilon,t^\prime}}_s\big)\,\mathcal{H}g\big(s,p_s,q_s;\xi^{\varepsilon,t^\prime}_s,b^{\varepsilon,t^\prime}_s\big)ds\bigg].
   \end{equation*}
   In the same way as before, we get 
   \begin{align*} 
   -\varepsilon&\le \lim_{t'\downarrow t}\Eop^{t,x,p,q}\bigg[\frac{1}{t'-t}\int_t^{t'} f\big(s,X^{\xi^{\varepsilon,t^\prime},b^{\varepsilon,t^\prime}}_s\big)\,\mathcal{H}g\big(s,p_s,q_s;\xi^{\varepsilon,t^\prime}_s,b^{\varepsilon,t^\prime}_s\big)ds\Ind{t'<\tau}\bigg] \\
   &\le\lim_{t'\downarrow t}\Eop^{t,x,p,q}\bigg[\frac{1}{t'-t}\int_t^{t'}\! f\big(s,X^{\xi^{\varepsilon,t^\prime},b^{\varepsilon,t^\prime}}_s\big)\!\inf_{(\xi,b)\in[-K,K]\times[0,1]}\!\mathcal{H}g\big(s,p_s,q_s;\xi,b\big)ds\Ind{t'<\tau}\bigg].
   \end{align*}
   %	Making use of Lemma~\ref{leDir:gprop}~(vii) and~(ix), we obtain as before
   %	\begin{align*}
   %	&\Eop^{t,x,p,q}\bigg[\frac{1}{t'-t}\int_t^{t'} f\big(s,X^{\xi^{\varepsilon,t^\prime},b^{\varepsilon,t^\prime}}_s\big)\,\inf_{(\xi,b)\in[-K,K]\times[0,1]}\mathcal{H}g\big(s,p_s,q_s;\xi,b\big)ds\Ind{t'<\tau}\bigg]\\
   %	&\le \Eop^{t,x,p,q}\bigg[\frac{1}{t'-t}\int_t^{t'} \big|f\big(s,X^{\xi^{\varepsilon,t^\prime},b^{\varepsilon,t^\prime}}_s\big)\big|(K_6+K_4)ds\bigg] 	\le (K_4+K_6)\,K_1 <\infty.
   %	\end{align*}
   We can again interchange the limit and the infimum by the dominated convergence theorem which yields 
   \begin{equation*}
   -\varepsilon\le \Eop^{t,x,p,q}\bigg[\lim_{t'\downarrow t}\frac{1}{t'-t}\int_t^{t'} f\big(s,X^{\xi^{\varepsilon,t^\prime},b^{\varepsilon,t^\prime}}_s\big)\,\inf_{(\xi,b)\in[-K,K]\times[0,1]}\mathcal{H}g\big(s,p_s,q_s;\xi,b\big)ds\Ind{t'<\tau}\bigg].
   \end{equation*}
   Thus the same conclusion can be draw as above, i.e. 
   \begin{equation*}
   -\varepsilon \le f(t,x)\inf_{(\xi,b)\in[-K,K]\times[0,1]}\mathcal{H}g(t,p,q;\xi,b)
   \end{equation*}
   at those point where $Dg(s,p,q)$ exists. According to the negativity of $f$ and the arbitrariness of $\varepsilon>0$, we get, by $\varepsilon\downarrow0$,
   \begin{equation*}
   0 \ge \inf_{(\xi,b)\in[-K,K]\times[0,1]}\mathcal{H}g(t,p,q;\xi,b)
   \end{equation*}
   at those point where $Dg(s,p,q)$ exists.
   By the same way as before, we obtain in the case of no differentiability of $g$ w.r.t.\ $t$ and $p_j$, $j=1,\ldots,m$, that 
   \begin{equation*}
   0 \ge \inf_{(\xi,b)\in[-K,K]\times[0,1]}\mathcal{L}g(t,p,q;\xi,b) + \inf_{\varphi\in\partial^C g_q(t,p)}\bigg\{\varphi_0+\sum_{j=1}^m \varphi_j\,p_j\bigg(\sum_{k=1}^m\lambda_k\,p_k -\lambda_j\bigg)\bigg\}.
   \end{equation*}
   Summarizing, we have equality in the previous expression.
   %	\begin{equation*}
   %	0 = \inf_{(\xi,b)\in[-K,K]\times[0,1]}\mathcal{L}g(t,p,q;\xi,b) + \inf_{\varphi\in\partial^C g_q(t,p)}\bigg\{\varphi_0+\sum_{j=1}^m \varphi_j\,p_j\bigg(\sum_{k=1}^m\lambda_k\,p_k -\lambda_j\bigg)\bigg\}.
   %	\end{equation*}
   %	With the same arguments as in the proof of Theorem~\ref{th:veriHJB}, we get
   %	\begin{equation*}
   %	\inf_{(\xi,b)\in\R\times[0,1]}\mathcal{L}g(t,p,q;\xi,b) = \inf_{(\xi,b)\in[-K,K]\times[0,1]}\mathcal{L}g(t,p,q;\xi,b)
   %	\end{equation*}
   %	and that
   %	\begin{equation*}
   %	g(t,p,q) = \inf_{(\xi,b)\in U[t,T]} g^{\xi,b}(t,p,q) = \inf_{(\xi,b)\in\UtT} g^{\xi,b}(t,p,q) = g(t,p,q).
   %	\end{equation*}
   %	Therefore, it follows from Lemma~\ref{leDir:gtildeprop}~(i),~(iii),~(iv),~(v),~(vi) and Theorem~\ref{thDir:veri} that
   %	\begin{equation*}
   %	V(t,x,p,q) = -e^{-\alpha x e^{r(T-t)}}g(t,p,q),\quad (t,x,p,q)\in\domainVDir,
   %	\end{equation*}
   The optimality of $(\xi^{\star},b^{\star})$  follows as in the proof of Theorem \ref{thDir:veri}.
\end{proof}

\begin{proof}[Proof of Lemma \ref{leDir:propg2}]
	\begin{enumerate}
		\item Follows  by conditioning.
	\item We observe that 
   \begin{equation*}
   g^{\xi,b}(t,p,q) = \sum_{k=1}^m \Pop^{t,p,q}(\Lambda=\lambda_k)\int h(\lambda_k,\tilde\alpha)\Pop^{t,p,q}(\bar\alpha\in d\tilde\alpha)
   \end{equation*}
   with
\begin{align*}
&h(\lambda,\tilde{\alpha}) \defeq \Eop\bigg[\exp\bigg\{-\int_t^T \alpha\, e^{r(T-s)}\big((\mu-r)\,\xi_s+c(b_s)\big)ds\\
&\quad -\int_t^T\alpha\,\sigma\, e^{r(T-s)}\xi_s dW_s  + \alpha\sum_{n=1}^{N_{T-t}} b_{T_n}\,e^{r(T-T_n)}\idsum Y_n\ind{Z_n}(i)\bigg\}\Bigg|\Lambda=\lambda,\alphabar=\tilde{\alpha}\bigg],
\end{align*}
and $\Pop^{t,p,q}(\Lambda=\lambda_k)=p_k$ and $\Pop^{t,p,q}(\alphabar\in\dif\tilde{\alpha})=f_{\bar{\beta}}(\tilde{\alpha}| q)\dif\tilde{\alpha}$, where
$f_{\bar{\beta}}(\tilde{\alpha}| q)$ denotes the posterior density function of $\tilde{\alpha}$ given $q_t=q$, compare Theorem~\ref{th:Dirconjugated}.
That is,
\begin{equation*}
f_{\bar{\beta}}(\tilde{\alpha}| q) =  \frac{\Gamma\big(\Esum(\beta_E+q_E)\big)}{\Eprod\Gamma\big(\beta_E+q_E\big)} \Eprod \alpha_E^{\beta_E+q_E-1},\quad\tilde{\alpha}=(\alpha_E)_{E\subset\mathbb{D}}\in\mathring{\Delta}_\ell.
\end{equation*}
Consequently, the statement holds if
\begin{equation*}
\Dsum\fracbetaq f_{\bar{\beta}}(\tilde{\alpha}| v(q,D)) = f_{\bar{\beta}}(\tilde{\alpha}| q).
\end{equation*}
Indeed, using $\Gamma(n+1)=n\,\Gamma(n)$ for all $n\in\N$, we have for any $\tilde{\alpha}=(\alpha_E)_{E\subset\mathbb{D}}\in\mathring{\Delta}_\ell$
\begin{align*}
&\Dsum\fracbetaq f_{\bar{\beta}}(\tilde{\alpha}| v(q,D)) \\
=&\Dsum\fracbetaq \frac{\Gamma\big(\Esum(\beta_E+q_E)+1\big)}{\Gamma\big(\beta_D+q_D+1\big)\EDDprod\Gamma\big(\beta_E+q_E\big)}\alpha_D^{\beta_D+q_D}\! \EDDprod \alpha_E^{\beta_E+q_E-1} \\
=&\Dsum\frac{\beta_D+q_D}{\Esum (\beta_E+q_E)}\frac{\big(\Esum(\beta_E+q_E)\big)\Gamma\big(\Esum(\beta_E+q_E)\big)}{\big(\beta_D+q_D\big)\Eprod\Gamma\big(\beta_E+q_E\big)} \alpha_D\Eprod \alpha_E^{\beta_E+q_E-1}\\
=&s\frac{\Gamma\big(\Esum(\beta_E+q_E)\big)}{\Eprod\Gamma\big(\beta_E+q_E\big)} \Eprod \alpha_E^{\beta_E+q_E-1}\Dsum\alpha_D =f_{\bar{\beta}}(\tilde{\alpha}| q),
\end{align*}
since $\Dsum\alpha_D=1$.	\qedhere
%\item We can conclude the assertion by using~\eqref{eqDir:separation} and the convexitiy of $V$ w.r.t.\ $p$, compare Lemma~\ref{leDir:propV}~(b). 
	\end{enumerate}
\end{proof}

% We observe that 
%\begin{equation*}
%g^{\xi,b}(t,p,q) = \sum_{k=1}^m\Pop^{t,p,q}(\Lambda=\lambda_k)\int_{\Delta_\ell} h(\lambda_k,\tilde{\alpha})\,\Pop^{t,p,q}(\alphabar\in\dif\tilde{\alpha})
%\end{equation*}

%\newpage

%\section{The Problem with Discrete Marginals}\label{sec:model}

%\bibliographystyle{apalike} % Bibliography Harvard style

%\bibliography{bibliography}

%{\em Acknowledgement:} The authors would like to thank  an anonymous  referee for helpful comments which led to some interesting remarks.

\end{document}